\documentclass[12pt]{amsart}

\usepackage{graphicx}
\usepackage{amsmath}
\usepackage{amscd}
\usepackage{amsfonts}
\usepackage{amssymb}
\usepackage[dvipsnames]{xcolor}
\usepackage{fullpage}
\usepackage{mathtools}
\usepackage[hypertexnames=false, colorlinks, citecolor=red, linkcolor=blue, urlcolor=red, bookmarksopen]{hyperref}
\usepackage{todonotes}
\usepackage{tikz-cd}
\usepackage{enumitem}
\usepackage{amsthm}

\numberwithin{equation}{section}

\newtheorem{theorem}{Theorem}[section]
\newtheorem{lemma}[theorem]{Lemma}
\newtheorem{proposition}[theorem]{Proposition}

\newtheorem{corollary}[theorem]{Corollary}

\theoremstyle{definition}
\newtheorem{example}[theorem]{Example}
\newtheorem{remark}[theorem]{Remark}
\newtheorem{definition}[theorem]{Definition}

\newcommand{\be}{\begin{equation}}
	\newcommand{\ee}{\end{equation}}
\newcommand{\bes}{\begin{equation*}}
	\newcommand{\ees}{\end{equation*}}

\newcommand{\cB}{\mathcal{B}}
\newcommand{\cC}{\mathcal{C}}

\newcommand{\cE}{\mathcal{E}}

\newcommand{\cH}{\mathcal{H}}
\newcommand{\cI}{\mathcal{I}}
\newcommand{\cJ}{\mathcal{J}}

\newcommand{\cL}{\mathcal{L}}

\newcommand{\cR}{\mathcal{R}}
\newcommand{\cS}{\mathcal{S}}

\newcommand{\cU}{\mathcal{U}}
\newcommand{\cV}{\mathcal{V}}

\newcommand{\cX}{\mathcal{X}}
\newcommand{\cY}{\mathcal{Y}}

\newcommand{\bB}{\mathbb{B}}
\newcommand{\bC}{\mathbb{C}}
\newcommand{\bD}{\mathbb{D}}

\newcommand{\bF}{\mathbb{F}}
\newcommand{\bM}{\mathbb{M}}
\newcommand{\bN}{\mathbb{N}}

\newcommand{\lip}{\langle}
\newcommand{\rip}{\rangle}
\newcommand{\ip}[1]{\lip #1 \rip}

\newcommand{\ol}{\overline}

\newcommand{\Rep}{\operatorname{Rep}}

\DeclareMathOperator*{\sotlim}{\textsc{sot}--lim}

\newcommand{\Aut}{\operatorname{Aut}}
\newcommand{\diag}{\operatorname{diag}}

\newcommand{\id}{\operatorname{id}}

\newcommand{\spn}{\operatorname{span}}

\newcommand{\fB}{{\mathfrak{B}}}
\newcommand{\fC}{{\mathfrak{C}}}
\newcommand{\fD}{{\mathfrak{D}}}

\newcommand{\fV}{{\mathfrak{V}}}
\newcommand{\fW}{{\mathfrak{W}}}
\newcommand{\fX}{{\mathfrak{X}}}

\newcommand{\foral}{\text{ for all }}
\newcommand{\qand}{\quad\text{and}\quad}

\newcommand{\AND}{\text{ and }}

\begin{document}


\title{\textbf{Weak-* and completely isometric structure of noncommutative function algebras}}

\author{Jeet Sampat}
\address{Department of Mathematics\\
Technion --- Israel Institute of Technology\\
Haifa, Israel}
\email{sampatjeet@campus.technion.ac.il}
\author{Orr Moshe Shalit}
\address{Department of Mathematics\\
Technion --- Israel Institute of Technology\\
Haifa, Israel}
\email{oshalit@technion.ac.il}

\subjclass[2010]{46L52, 47B32, 47L80}

\thanks{The second author was partially supported by Israel Science Foundation Grant no. 431/20.}

\begin{abstract}
We study operator algebraic and function theoretic aspects of algebras of bounded nc functions on subvarieties of the nc domain determined by all levels of the unit ball of an operator space (nc operator balls). 
Our main result is the following classification theorem: under very mild assumptions on the varieties, two such algebras $H^\infty(\fV)$ and $H^\infty(\fW)$ are completely isometrically and weak-* isomorphic if and only if there is a nc biholomorphism between the varieties. 
For matrix spanning homogeneous varieties in injective operator balls, we further sharpen this equivalence, showing that there exists a linear isomorphism between the respective balls that maps one variety onto the other; in general, we show, the homogeneity condition cannot be dropped. 
We highlight some difficulties and open problems, contrasting with the well studied case of row ball. 
\end{abstract}

\maketitle


\section{Introduction} \label{sec:intro}

The theory of noncommutative (nc) functions was born in the 1970s as part of a natural conceptual generalization of the functional calculus for commuting tuples of operators \cite{Taylor-frame,Taylor-ncfunc}. 
Around the turn of the millennium the theory was rediscovered and revived with applications in operator theory, systems/control theory and free probability in mind (see, e.g., \cite{ball2006conservative, helton2007linear, Popescu-funcI,Voiculescu-quest1}), and today is a well established area \cite{agler2020operator,ball2021noncommutative,KVV14} flourishing into multiple research directions \cite{AM16, AglMcCYng, AHKM, BMV16, BMV18, dmitrieva2024bundles, evert2017circular, JMS-bso, JMS21, pirkovskii2019holomorphic, PV18}. 

This paper continues our ongoing study of the isomorphism problem for operator algebras of bounded nc functions on subvarieties of nc operator balls \cite{SSS18,SSS20,SampatShalit25}, which has grown out of earlier endeavors to classify operator algebras in terms of geometric invariants \cite{DRS11,DRS15,KS19,SS09}. 
Besides being a compelling class of concrete operator algebras which is interesting to study in its own right, the classification program for these algebras has driven progress in purely nc function theoretic questions.
For example, maximal principle, extension theorems and clarification of the similarity envelope in the above cited papers, or applications to iteration theory in \cite{BS+,Sha18}.
In the following paragraphs we shall describe what we achieve in this paper and how this differs from what was done previously. The reader who is not fluent in the parlance of nc functions can refer to the next section where the terms and notations we used are leisurely explained.

The guiding problem of this paper is the classification of the algebras of $H^\infty(\fV)$ of bounded nc functions on a subvariety $\fV$ inside a nc operator ball $\bD_Q$ (see Definition \ref{def:nc.op.ball}). 
Inspired by the main results of \cite{SSS18}, one may guess that given two varieties $\fV_i \subseteq \bD_{Q_i}$, ($i=1,2$), the algebras $H^\infty(\fV_1)$ and $H^\infty(\fV_2)$ are completely isometrically isomorphic if and only if there is a nc biholomorphism between $\fV_2$ and $\fV_1$. 
However, since we are working in the setting of general nc operator balls, rather than in the thoroughly studied and well understood setting of the row ball, and, moreover, since we are not restricting attention to homogeneous varieties, there are several challenges we shall need to overcome. 

The main goal of Section \ref{section:canon.w-*.top.H^infty(D_Q)} is to show that the algebras $H^\infty(\fV)$ of bounded nc function are dual spaces, supplied with a natural weak-* topology which is the unique predual for which point evaluations are continuous. 
This is significant, because a key tool used in earlier works on algebras of bounded nc functions in the row ball was that these algebras come with a faithful representation on a natural nc RKHS, and so they inherited a weak-* topology. In the case of general nc operator balls it was shown in \cite[Theorem 2.4]{SampatShalit25} that there is no such natural representation. 

The utility of realizing $H^\infty(\fV)$ as a dual space is discussed in Section \ref{subsec:fin.dim.reps}, where it is explained that the points in a nc variety $\fV$ correspond to weak-* representations, through the association of every $X \in \fV$ with the point evaluation 
\[
\Phi_X \colon f \mapsto f(X). 
\]
These are not all weak-* continuous representations, because some $X \in \partial \fV$ do give rise to weak-* continuous point evaluations. Moreover, it is not clear whether over a point $X \in \fV$ there might be fibered representations that are not weak-* continuous (see Remark \ref{rem:pure.tuples}). 
Given these two difficulties, the variety $\fV$ becomes a more interesting invariant of the algebras $H^\infty(\fV)$ than it would have been otherwise.

In the exploratory Section \ref{sec:rep_ball}, we share our attempts to overcome the second difficulty mentioned above and highlight the difficulties encountered. We exhibit a new noncommutative phenomenon: the remainder term in the nc Taylor-Taylor formula of a {\em bounded nc function} might be unbounded on the ball. 
This obstruction --- which arises already at the remainder of order $1$ and at the scalar level --- is the reason why we need to restrict our classification scheme to weak-* continuous isomorphisms. It leaves us with many interesting questions and research directions. We close Section \ref{sec:rep_ball} with some partial positive results. 

The main results of this paper are obtained in the second and third parts of Section \ref{sec:class.alg.H(V)}. 
In Theorem \ref{thm:gen.class.thm} we show, under the assumption that $\fV_2$ contains a scalar point, that there exists a weak-* continuous completely isometric isomorphism $\varphi \colon H^\infty(\fV_1) \to H^\infty(\fV_2)$ if and only if there is a nc biholomorphism $F \colon \fV_2 \to \fV_1$; further, such an isomorphism is implemented as $\varphi(f) = f \circ F$. 
To show that $F$ maps $\fV_2$ into $\fV_1$ (and does not send interior points to the boundary) we first prove a function theoretic {\em boundary value principle} (Theorem \ref{thm:bdy.val.prin}), which is a dichotomy saying that the range of a nc map from a subvariety of a nc operator ball into another nc operator ball is either contained entirely in the interior of the ball or contained entirely in the boundary; this refines an earlier variant from \cite{SampatShalit25}. 

Our sharpest classification result is obtained for matrix spanning homogeneous varieties inside injective nc operator balls (see Section \ref{subsec:inj.nc.op.balls} for definitions). 
We state it here for convenience: 

\begin{theorem}[Theorem \ref{thm:sharpest.classification.thm}]
For $i = 1,2$, let $\fV_i \subseteq \bD_{Q_i}$ be matrix-spanning homogeneous subvarieties of some injective nc operator balls. Then, the following are equivalent.

\begin{enumerate}[leftmargin=*]
\item There is a weak-* continuous completely isometric isomorphism $\varphi \colon H^\infty(\fV_1) \to H^\infty(\fV_2)$.

\item There is a nc biholomorphism of $\fV_2$ onto $\fV_1$.

\item There is a nc biholomorphism $F \colon \bD_{Q_2} \to \bD_{Q_1}$ such that $F(\fV_2) = \fV_1$.

\item There is a linear isomorphism $L \colon \bD_{Q_2} \to \bD_{Q_1}$ such that $L(\fV_2) = \fV_1$.

\item There is a completely isometric isomorphism $\widetilde{\varphi} \colon A(\fV_1) \to A(\fV_2)$.
\end{enumerate}
\end{theorem}

The point is that if two algebras of bounded nc functions on varieties are weak-* continuous and completely isometrically isomorphic, then not only are the varieties biholomorphic (as we know from Theorem \ref{thm:gen.class.thm}) --- but the entire balls must also be biholomorphic and, in fact, one can choose a linear biholomorphism. 
This means that the corresponding operator spaces must be completely isometric. 
In Example \ref{ex:cant.drop.homo} we show that the homogeneity assumption cannot be dropped. 

\section{NC Function Theory Background} \label{sec:nc.fnc.thry.bkgd}

\subsection{\textbf{NC sets and the nc universe}} \label{subsec:nc.sets.nc.universe}

Given any $m, n, d \in \bN$, let $M_{m \times n}^d$ denote the space of all $d$-tuples of $m \times n$ matrices with complex entries. For convenience of notation, we write
\begin{equation*}
M_{m \times n} := M_{m \times n}^1 \qand M_n^d := M_{n \times n}^d.
\end{equation*}
Using the operator norm $\left\| \cdot \right\|$ on $M_{m \times n}$, we endow $M_{m \times n}^d$ with the topology induced by the \emph{supremum norm} $\left\|\cdot\right\|_{\infty}$:
\begin{equation} \label{eqn:sup.norm.M_n^d}
\|X\|_{\infty} := \max_{1 \leq j \leq d} \|X_j\|, \foral X = (X_1, \dots, X_d) \in M_{m \times n}^d. 
\end{equation}
There is also a natural action of $GL_n(\bC)$ on each $M_n^d$ via $\cdot \colon GL_n(\bC) \times M_n^d \to M_n^d$, defined as
\begin{equation*}
S \cdot X := (S^{-1} X_1 S, \dots, S^{-1} X_d S), \foral X = (X_1, \dots, X_d) \in M_n^d \AND S \in GL_n(\bC).
\end{equation*}

\begin{definition} \label{def:nc.universe.sets.domains.functions}
Fix $d, e \in \bN$.
\begin{enumerate}[itemindent=*, leftmargin=0mm, itemsep=2mm]

\item  The \emph{$d$-dimensional nc universe} is defined to be the graded union
\begin{equation*}
\bM^d := \bigsqcup_{n = 1}^\infty M_n^d \cong \bigsqcup_{n = 1}^\infty M_n(\bC^d) \cong \bigsqcup_{n = 1}^\infty \bC^d \otimes M_n.
\end{equation*}

\item A subset $\Omega \subseteq \bM^d$ is said to be a \emph{nc set} if it is closed under direct sums, i.e.,
\begin{equation*}
X, Y \in \Omega \implies X \oplus Y \in \Omega.
\end{equation*}

\item The collection of subsets $\Omega \subset \bM^d$ whose \emph{$n$-th level}, i.e., $\Omega(n) := \Omega \cap M_n^d$ is open in $M_n^d$ for each $n \in \bN$ forms a topology on $\bM^d$ called the \emph{disjoint union topology}.

\item A \emph{nc domain} $\Omega \subseteq \bM^d$ is a nc set which is open in the disjoint union topology and level-wise connected.

\item For any $\Omega \subseteq \bM^d$, we denote its closure in the disjoint union topology by $\overline{\Omega}$, i.e.,
\begin{equation*}
\overline{\Omega} := \bigsqcup_{n = 1}^\infty \overline{\Omega(n)}.
\end{equation*}

\item \label{item:bdd.nc.set} A subset $\Omega \subseteq \bM^d$ is said to be \emph{bounded} if the levels $\Omega(n)$ are uniformly bounded in $M_n^d$ under the supremum norm $\left\| \cdot \right\|_\infty$ (as in (\ref{eqn:sup.norm.M_n^d})).

\item \label{item:nc.functions} A function $F \colon \Omega \to \bM^e$ on a nc set $\Omega \subseteq \bM^d$ is said to be a \emph{nc function} if
\begin{enumerate}[itemindent=3mm, itemsep=2mm]

\item[\textbullet] $F$ is \emph{graded}:
\begin{equation*}
X \in \Omega(n) \implies F(X) \in M_n^e, \foral n \in \bN.
\end{equation*}

\item[\textbullet] $F$ \emph{respects direct sums}:
\begin{equation*}
X, Y \in \Omega \implies F(X \oplus Y) = F(X) \oplus F(Y).
\end{equation*}

\item[\textbullet] $F$ \emph{respects similarities}:
\begin{equation*}
S \in GL_n(\bC) \textrm{ and } X, \, S \cdot X \in \Omega(n) \implies F(S \cdot X) = S \cdot F(X), \foral n \in \bN.
\end{equation*}

\end{enumerate}


\end{enumerate}
\end{definition}

While we mostly work with nc functions on nc subsets of $\bM^d$, it is important to note that one can easily work in a setting more general than $d$-tuples of matrices. 
Given any two Hilbert spaces $\cR \AND \cS$, consider $\cL(\cR,\cS)$, i.e., the space of bounded linear operators from $\cR$ into $\cS$ equipped with the family $\left\{\left\| \cdot \right\|_n\right\}_{n \in \bN}$ of matrix-operator norms:
\begin{equation*}
\left\|\left[T_{j,k}\right]\right\|_n := \left\|\left[T_{j,k}\right]\right\|_{\cL(\cR^n,\cS^n)}, \foral \left[T_{j,k}\right] \in \cL(\cR^n,\cS^n) \AND n \in \bN.
\end{equation*}
We can then easily generalize Definition \ref{def:nc.universe.sets.domains.functions} for the graded union
\begin{equation} \label{eqn:L(R,S)}
\cL(\cR,\cS)_{nc} := \bigsqcup_{n = 1}^\infty \cL(\cR^n,\cS^n) \cong \bigsqcup_{n = 1}^\infty M_n(\cL(\cR,\cS)) \cong \bigsqcup_{n = 1}^\infty \cL(\cR,\cS) \otimes M_n.
\end{equation}
NC functions are similarly defined on subsets of $\sqcup_n M_n(E)$ for any vector space $E$, and one can also naturally define a disjoint union topology once $E$ is a topological vector space; finally, one can define norms and open balls once $E$ is an \emph{operator space}. We refer the reader to \cite[Section 1.2]{KVV14} and the discussion surrounding it for a general treatise on this topic. Henceforth, we use `nc set' and `nc function' liberally for generalizations of $\bM^d$, and we hope that the meaning will be clear to the reader.

\subsection{\textbf{Free nc polynomials}} \label{subsec:nc.poly.nc.function.algebra}

Let $\bF_d^+$ be the free unital semigroup in $d$ generators $\{1, \dots, d\}$, consisting of \emph{free words} $\alpha = \alpha_1 \dots \alpha_k$ of arbitrary \emph{size} $|\alpha| := k \in \bN$. For $d$ free noncommuting variables $Z = (Z_1, \dots, Z_d)$ and a given $\alpha \in \bF_d^+$, we write $Z^\alpha := Z_{\alpha_1} \dots Z_{\alpha_k}$ for the \emph{nc monomial corresponding to the free word $\alpha$}. We allow every $Z^\alpha$ to be evaluated at any $d$-tuple of operators $T = (T_1, \dots, T_d) \in B(\cH)^d$ on a Hilbert space $\cH$ via $T^\alpha = T_{\alpha_1} \dots T_{\alpha_k}$.

\begin{definition} \label{def:free.nc.poly}
\begin{enumerate}[itemindent=*, leftmargin=0mm, itemsep=2mm]

\item \label{item:free.nc.polynomial} A \emph{free nc polynomial} (or simply, \emph{polynomial}) is a formal sum
\begin{equation*}
P(Z) := \sum_{\alpha \in \bF_d^+} c_\alpha Z^\alpha,
\end{equation*}
where all but finitely many $c_\alpha \in \bC$ are equal to $0$.

\item \label{item:nc.polynomial.algebra} We define $\bC \langle Z \rangle := \bC \langle Z_1, \dots, Z_d \rangle$ to be the \emph{algebra of free nc polynomials in $d$ free noncommuting variables $(Z_1,\dots,Z_d)$}.

\item \label{item:homogeneous.polynomial} A free nc polynomial $P$ is said to be \emph{homogeneous of degree $p$} if, for some $p \in \bN$, we have
\begin{equation*}
P(\lambda Z) = \lambda^p P(Z), \foral \lambda \in \bC.
\end{equation*}

\end{enumerate}
\end{definition}

Clearly, every $P \in \bC \langle Z \rangle$ can be evaluated on a given $X \in \bM^d$, and therefore can be realized as a function $P \colon \bM^d \to \bM^1$, which is readily seen to be a nc function. More generally, let $\cL(\cR,\cS)_{nc}$ be as in (\ref{eqn:L(R,S)}) and define an \emph{operator-valued free nc polynomial} $Q \colon \bM^d \to \cL(\cR,\cS)_{nc}$ to be the nc map given by a finite sum
\begin{equation*}
Q(Z) := \sum_{\alpha \in \bF_d^+} Q_\alpha Z^\alpha,
\end{equation*}
so that all but finitely many $Q_\alpha \in \cL(\cR,\cS)$ are $0$. Note that $Q$ is interpreted functionally via
\begin{equation*}
Q(X) = \sum_{\alpha \in \bF_d^+} Q_\alpha \otimes X^\alpha, \foral X \in \bM^d.
\end{equation*}

\subsection{\textbf{NC operator balls}} \label{subsec:nc.op.balls}

Two important examples of nc domains are the following.
\begin{enumerate}[itemindent=*, leftmargin=0mm, itemsep=2mm]

\item \label{item:nc.unit.row-ball} The \emph{nc unit row-ball} $\fB_d$, given by
\begin{equation} \label{eqn:nc.unit.row-ball}
\fB_d := \left\{ X \in \bM^d : \left \| \sum_{j = 1}^d X_j X_j^* \right \| < 1 \right\}.
\end{equation}

\item \label{item:nc.unit.polydisk} The \emph{nc unit polydisk} $\fD_d$, given by
\begin{equation} \label{eqn:nc.unit.polydisk}
\fD_d := \left\{ X \in \bM^d : \|X\|_\infty < 1 \right\}.
\end{equation}

\end{enumerate}
Clearly, $\fB_d$ and $\fD_d$ are nc domains. In this paper, we wish to study nc functions over a large class of nc domains that generalize $\fB_d$ and $\fD_d$.

\begin{definition} \label{def:nc.op.ball}
Let $\cL(\cR,\cS)_{nc}$ be as in (\ref{eqn:L(R,S)}), and consider a $d$-dimensional operator subspace $\cE \subseteq \cL(\cR,\cS)$. Let $\{Q_1, \dots, Q_d\}$ be a basis for $\cE$ and define an \emph{injective operator-valued linear polynomial} $Q \colon \bM^d \to \cL(\cR,\cS)_{nc}$ given by
\begin{equation} \label{eqn:op.val.lin.nc.poly}
Q(Z) = \sum_{j = 1}^d Q_j Z_j.
\end{equation} For such a $Q$, the corresponding \emph{nc operator ball} is defined to be
\begin{equation} \label{eqn:nc.op.ball}
\bD_Q := \left\{ X \in \bM^d : \left\|Q(X)\right\| < 1 \right\}.
\end{equation}
\end{definition}

Given any $n \in \bN$, we identify $\bD_Q(n)$ with $B_1(M_n(\cE))$ -- the open unit ball of $M_n(\cE)$ -- via
\begin{equation} \label{eqn:D_Q=B_1(M_n(E))}
\bD_Q \ni X \leftrightarrow Q(X) \in B_1(M_n(E)).
\end{equation}
Note that $Q$ is, in fact, a linear isomorphism between $M_n^d$ and $M_n(\cE)$ for each $n \in \bN$. The identification in (\ref{eqn:D_Q=B_1(M_n(E))}) shows that, despite our choice of the basis $\{ Q_1, \dots, Q_d \}$, the nc operator ball $\bD_Q$ is uniquely identified by the underlying operator space $\cE$ up to a linear change of coordinates. Equation (\ref{eqn:D_Q=B_1(M_n(E))}) also gives us that $\bD_Q(n)$ is non-empty, open and bounded, by the virtue of being a preimage of the open unit ball $B_1(M_n(\cE))$ under a linear isomorphism. This shows that $\bD_Q$ is non-empty, open with respect to the disjoint union topology, and bounded at each level. We can actually show that $\bD_Q$ is open in a stronger topology, and it is also bounded (as in Definition \ref{def:nc.universe.sets.domains.functions} (6)) by showing that $\bD_Q$ satisfies a certain convexity condition.

\begin{definition} \label{def:uniform.nc.topology}
For any $X \in M_n^d, \, n \in \bN$ and $r > 0$, we define the \emph{nc ball $\cB_{nc}(X,r)$ of radius $r$ centered at $X$} via
\begin{align*}
\cB_{nc}(X,r) := &\bigsqcup_{m = 1}^\infty \cB\left(\bigoplus_{k = 1}^m X, r\right) \\
= &\bigsqcup_{m = 1}^\infty \left\{ Y \in M_{mn}^d : \left\| Y - \bigoplus_{k = 1}^m X \right\|_\infty < r \right\}.
\end{align*}
The collection of all nc balls forms a topology on $\bM^d$, which we call the \emph{uniformly-open nc topology on $\bM^d$}. An open set with respect to this topology is called a \emph{uniformly-open nc set}.
\end{definition}

\begin{definition} \label{def:matrix.convexity}
A nc set $\Omega \subseteq \bM^d$ is called \emph{matrix convex} if for any $X = (X_1, \dots, X_d) \in \Omega(n)$ and any \emph{unital completely positive (UCP) map} $\phi \colon M_n \to M_k$, we have
\begin{equation} \label{eqn:matrix.convexity.def}
\phi(X) := \left(\phi(X_1), \dots, \phi(X_d)\right) \in \Omega(k).
\end{equation}
\end{definition}

\begin{proposition} \label{prop:boundedness.D_Q}
$\bD_Q$ is a matrix convex nc set that is bounded and uniformly-open.
\end{proposition}

\begin{proof}
Let $X = (X_1, \dots, X_d) \in \bD_Q(n)$ and $\phi \colon M_n \to M_k$ be a UCP map. Note that
\begin{equation} \label{eqn:matrix.convexity.D_Q.ineq}
\left\| Q(\phi(X)) \right\| = \left\| \left[\id_{\cL(\cR,\cS)} \otimes \phi \right] Q(X) \right\| \leq \|Q(X)\| < 1.
\end{equation}
The first inequality in (\ref{eqn:matrix.convexity.D_Q.ineq}) follows from \cite[Proposition 2.1.1]{Pis03} and \cite[Proposition 3.6]{Pau02}. Hence, we get (\ref{eqn:matrix.convexity.def}) and it follows that $\bD_Q$ is a matrix convex nc set.

To show boundedness of $\bD_Q$, first note that $\bD_Q(1)$ is bounded by the discussion following (\ref{eqn:D_Q=B_1(M_n(E))}). Since $\bD_Q(1)$ has a non-empty interior, boundedness of $\bD_Q$ follows from \cite[Lemma 3.4]{davidson2017dilations}. For completeness, let us show directly that if $\|x\|_\infty <r$ for all $x \in \bD_Q(1)$, then $\|X\|_\infty < 2r$ for all $X \in \bD_Q$ (the argument works for any matrix convex set, and the constant $2$ is optimal). 
Assume therefore, that
\begin{equation} \label{eqn:choice.of.r}
\bD_Q(1) \subset r \bD^d,
\end{equation}
where $\mathbb{D}$ is the open unit disk in $\mathbb{C}$.
Let $X = (X_1, \dots, X_d) \in M_n^d$ such that $\|X_j\| \geq 2r$ for some $1 \leq j \leq d$. We can find a unit vector $v \in \bC^n$ such that
\begin{equation} \label{eqn:choice.of.v}
|\ip{X_j v, v}_{\bC^n}| \geq r
\end{equation}
and use it to define a \emph{state} $\phi_v \colon M_n \to M_1$ via
\begin{equation*}
\phi_v(Y) := \ip{Yv,v}_{\bC^n}, \foral Y \in M_n.
\end{equation*}
By (\ref{eqn:choice.of.v}), we get that
\begin{equation*}
|\phi_v(X_j)| = |\ip{X_j v , v}_{\bC^n}| \geq r,
\end{equation*}
and hence $\phi_v(X) \not\in r \bD^d$, and so $\phi_v(X) \notin \bD_Q(1)$. 
Since every state is also UCP (see \cite[Proposition 3.8]{Pau02}) the matrix convex set $\bD_Q$ is closed under application of states, and we conclude that $X \notin \bD_Q$. 

Lastly, it has already been established in \cite[Section 3.2]{SampatShalit25}  that $\bD_Q$ is a uniformly-open nc set. This completes the proof.
\end{proof}

\begin{example}
Both $\fB_d$ and $\fD_d$ are nc operator balls.
\begin{enumerate}[itemindent=*, leftmargin=0mm, itemsep=2mm]

\item $\cE = \textit{row operator space } \bC^d_{\text{row}} \AND Q(Z) = \begin{bmatrix}
    Z_1 & \dots & Z_d
\end{bmatrix} \implies \bD_Q = \fB_d$.

\item $\cE = \textit{minimal operator space } \operatorname{min}(\ell^\infty(\bC^d))\AND Q(Z) = \diag(Z_1, \dots, Z_d) \implies \bD_Q = \fD_d$.

\end{enumerate}
\end{example}

One might wish to generalize further and consider \emph{nc homogeneous polyhedrons} $\bD_Q$ given by operator-valued homogeneous free nc polynomials $Q$ of higher degree but, as was pointed out in \cite[Remark 1.2]{SampatShalit25}, such homogeneous bounded nc polyhedra are unbounded. 
Thus, if we wish to consider algebras of bounded nc functions, nc operator balls are a natural setting. 
That said, the reader should keep in mind that there are other nc domains of interest besides nc operator balls (such as free spectrahedra \cite{AHKM,evert2017circular} and nc polydomains \cite{popescu2020brerezin}).

\subsection{\textbf{NC function algebras}} \label{subsec:nc.function.algebras}

A striking feature of nc function theory is that a mild local boundedness condition on a nc function $F$ is enough to guarantee holomorphicity of $F$. In fact, a bounded nc function is always holomorphic in all nc topologies of interest. This motivates us to define the algebra $H^\infty(\bD_Q)$ of bounded nc functions on $\bD_Q$, i.e.,
\begin{equation*}
H^\infty(\bD_Q) := \left\{ F \colon \bD_Q \to \bM^1 : F \text{ is nc and } \|F\|_\infty := \sup_{X \in \bD_Q} \|F(X)\| < \infty \right\}.
\end{equation*}
The protagonist in \cite{SampatShalit25} was the separable subalgebra $A(\bD_Q) \subsetneq H^\infty(\bD_Q)$ defined as
\begin{equation*}
A(\bD_Q) := \overline{\bC \langle Z \rangle}^{\left\|\cdot\right\|_\infty}.
\end{equation*}

By Proposition \ref{prop:boundedness.D_Q}, we know that $\bD_Q$ is uniformly-open. Applying \cite[Theorem 7.21]{KVV14} therefore gives us at once that every $F \in H^\infty(\bD_Q)$ has a \emph{nc power-series expansion (around $0$)}, i.e., there exist $\left\{ c_\alpha \right\}_{\alpha \in \bF_d^+} \subset \bC$ such that
\begin{equation*}
F(Z) = \sum_{\alpha \in \bF_d^+} c_\alpha Z^\alpha.
\end{equation*}
The nc power-series converges absolutely and uniformly on $r \bD_Q$ for every $r < 1$. The nc power-series expansion also gives us a formal \emph{homogeneous expansion} $F = \sum_k F_k$, where
\begin{equation} \label{eqn:kth.homogeneous.term}
F_k(Z) := \sum_{|\alpha| = k} c_\alpha Z^\alpha, \foral k \in \bN \cup \{0\}.
\end{equation}

We exploited properties of the homogeneous expansion in order to prove a \emph{homogeneous nc Nullstellensatz} for both $H^\infty(\bD_Q)$ and $A(\bD_Q)$ \cite[Proposition 3.4]{SampatShalit25}, which further enabled us to classify $A(\bD_Q)$ and its quotients up to completely isometric isomorphisms. As for $H^\infty(\bD_Q)$, we obtained a homogeneous nc Nullstellensatz but not a complete classification result. One of the main challenges we had was working with the \emph{bounded pointwise topology} on $H^\infty(\bD_Q)$, and the lack of a good understanding of the finite dimensional bounded pointwise continuous representations of $H^\infty(\bD_Q)$.

\begin{definition} \label{def:bdd.pt-wise.topology}
The bounded pointwise topology on $H^\infty(\bD_Q)$ is the strongest locally convex topology in which a bounded net $\left\{ F_{\kappa} \right\}_{\kappa \in K} \subset H^\infty(\bD_Q)$ converges to some $F \in H^\infty(\bD_Q)$ if and only if
\begin{equation*}
F_\kappa (X) \xrightarrow{\left\|\cdot\right\|_\infty} F(X), \foral X \in \bD_Q.
\end{equation*}
\end{definition}

In Section \ref{section:canon.w-*.top.H^infty(D_Q)} we will show that $H^\infty(\bD_Q)$ is a dual space and that the topology of pointwise convergence on $H^\infty(\bD_Q)$ coincides with the weak-* topology on bounded sets. 

\subsection{\textbf{NC subvarieties and ideals}} \label{subsec:nc_subvarieties_and_ideals}


\begin{definition} \label{def:subvarieties}

\begin{enumerate}[itemindent=*, leftmargin=0mm, itemsep=2mm]

\item A \emph{nc (analytic) subvariety} $\fV \subseteq \bD_Q$ is the common zero set of a collection of nc functions $S \subseteq H^\infty(\bD_Q)$, i.e.,
\begin{equation*}
\fV := V_{\bD_Q}(S) = \left\{ X \in \bD_Q : F(X) = 0, \foral F \in S \right\}.
\end{equation*}
The subvariety $\fV$ is said to be \emph{algebraic} if $S \subseteq \bC \langle Z \rangle$. 

\item The \emph{corresponding ideal} $I(\fV) \subseteq H^\infty(\bD_Q)$ is defined as
\begin{equation*}
I(\fV) := \left\{ F \in H^\infty(\bD_Q) : F(X) = 0, \foral X \in \fV \right\}.
\end{equation*}

\item $\fV$ is said to be \emph{homogeneous} if $\lambda \fV \subseteq \fV$ for each $\lambda \in \bD$.

\item An ideal $\cI$ is said to be \emph{homogeneous} if for each $F = \sum_k F_k \in \cI$, as in \eqref{eqn:kth.homogeneous.term}, we have $F_k \in \cI \foral k \in \bN \cup \{0\}$. Equivalently, $\cI$ is homogeneous if for each $F \in \cI$ and $0 < r < 1$, the function $F_r \colon X \mapsto F(rX)$ lies in $\cI$ \cite[Proposition 3.6]{SampatShalit25}.

\end{enumerate}

\end{definition}

\begin{remark} \label{remark:equiv.hom.var.ideal}
One can check that if $\fV$ is a homogeneous nc subvariety then $I(\fV)$ is homogeneous as well. Conversely, if $\cI$ is a homogeneous ideal then $V_{\bD_Q}(\cI)$ is also homogeneous.
\end{remark}

We now define the operator algebras $H^\infty(\fV)$ and $A(\fV)$ over a nc subvariety $\fV \subseteq \bD_Q$ as
\begin{align*}
H^\infty(\fV) &:= \left\{ f \colon \fV \to \bM^1 : f \text{ is nc and } \|f\|_\fV := \sup_{X \in \fV} \|f(X)\| < \infty \right\}, \\
A(\fV) &:= \overline{\left\{ p \big\vert_{\fV} : p \in \bC \langle Z \rangle \right\}}^{\left\|\cdot\right\|_{\fV}}.
\end{align*}
The linear map $R \colon H^\infty(\bD_Q) \to H^\infty(\fV)$ given by
\begin{equation*}
R(F) = F \big\vert_{\fV}, \foral F \in H^\infty(\bD_Q)
\end{equation*}
is a surjective complete contraction with $\ker(R) = I(\fV)$ and, moreover, the induced quotient map $\overline{R} \colon H^\infty(\bD_Q)/I(\fV) \to H^\infty(\fV)$ is a completely isometric isomorphism \cite[Theorem 4.8]{SampatShalit25}. For $A(\fV)$, the analogous statement was shown only under the assumption that $\fV$ is homogeneous. This identification, combined with the homogeneous nc Nullstellensatz, allowed us to classify quotients of $A(\bD_Q)$ by homogeneous ideals up to completely isometric isomorphisms. In this paper, we go beyond homogeneous subvarieties, and obtain a classification result for the quotient algebras $H^\infty(\bD_Q)/I(\fV) \cong H^\infty(\fV)$. To achieve this, we first show that the topology of pointwise convergence can be realized as the weak-* topology by identifying $H^\infty(\bD_Q)$ as a dual space.

\section{Canonical weak-* topology on \texorpdfstring{$H^\infty(\bD_Q)$}{H(D\_Q)}} \label{section:canon.w-*.top.H^infty(D_Q)}

Let $\cY$ be a Banach space and let $\cX \subseteq \cY^*$ be a subspace. 
Then, the canonical embedding of $\cY$ into $\cY^{**}$ turns $\cY$ into a space of bounded linear functionals on $\cX$. If under this embedding we have $\cY = \cX^*$, then we say that $\cX$ is a predual of $\cY$. 
We note that $\cX$ is a predual of $\cY$ if and only if it satisfies the following properties (see, e.g., \cite[Section 2]{DW11}):
\begin{enumerate}[itemindent=*, leftmargin=0mm, itemsep=2mm]

\item $\cX$ \emph{norms} $\cY$, i.e.,
\begin{equation} \label{eqn:B.norms.H^infty(D_Q)}
\sup \{ |\varphi(y)| : \|\varphi\| \leq 1, \, \varphi \in \cX \} = \|y\|_\cY, \foral y \in \cY.
\end{equation}

\item The closed unit ball $\overline{B_1(\cY)}$ is compact in the $\sigma(\cY, \cX)$ topology, i.e, the weakest topology on $\cY$ such that the functionals $\varphi \in \cX$ are continuous.

\end{enumerate}

We use this fact to obtain a canonical predual of $H^\infty(\bD_Q)$ for any $\bD_Q$. 
Let $Q$ and $\bD_Q$ be as in Definition \ref{def:nc.op.ball}. For any $n \in \bN$ and $X \in \bD_Q(n)$, define $\Phi_X \colon H^\infty(\bD_Q) \to M_n$ via
\begin{equation} \label{eqn:def.Phi_X}
\Phi_X(F) := F(X), \foral F \in H^\infty(\bD_Q).
\end{equation}
Then, for any $\eta \in M_n^*$, we get $\varphi_{X,\eta} \in (H^\infty(\bD_Q))^*$ given by
\begin{equation} \label{eqn:canonical.element.predual}
\varphi_{X,\eta} := \eta \circ \Phi_X.
\end{equation}
Define $\cX(\bD_Q) \subseteq (H^\infty(\bD_Q))^*$ to be the closed linear hull of all $\varphi_{X, \eta}$, i.e.,
\begin{equation} \label{eqn:def.predual}
\cX(\bD_Q) := \ol{\spn} \left\{ \varphi_{X,\eta} : X \in \bD_Q(n) \AND \eta \in M_n^*, \foral n \in \bN \right\}.
\end{equation}

\begin{theorem} \label{theorem:predual.H^infty(D_Q)}
$\cX(\bD_Q)$ is the unique predual $\cX$ of $H^\infty(\bD_Q)$ for which the point evaluation $\Phi_X$ given in \eqref{eqn:def.Phi_X} is $\sigma(H^\infty(\bD_Q), \cX)$ continuous for every $X \in \bD_Q$.
\end{theorem}

\begin{proof}
In order to show that $\cX(\bD_Q)$ is a predual of $H^\infty(\bD_Q)$, we need to show items $(1) \AND (2)$ as in the introduction of this Section for $\cX = \cX(\bD_Q)$ and $\cY = H^\infty(\bD_Q)$.
\begin{enumerate}[itemindent=*, leftmargin=0mm, itemsep=2mm]

\item Fix $F \in H^\infty(\bD_Q)$ and let $X \in \bD_Q(n)$ be given for some $n \in \bN$. Choose $\eta_X \in M_n^*$ so that
\begin{equation} \label{eqn:choice.of.eta}
|\eta_X (F(X))| = \| F(X) \| \AND \|\eta_X\| = 1.
\end{equation}
Then, we have
\begin{equation} \label{eqn:predual.check.1}
|\varphi_{X,\eta_X}(F)| = |\eta_X (F(X))| = \|F(X)\|.
\end{equation}
Combining (\ref{eqn:canonical.element.predual}) with (\ref{eqn:choice.of.eta}), we know that
\begin{equation} \label{eqn:norm.of.varphi_X,eta}
\left\|\varphi_{X, \eta_X}\right\| \leq \left\|\eta_X\right\| \left\| \Phi_X \right\| = \left\|\Phi_X\right\| \leq 1.
\end{equation}
Taking the supremum over every $X \in \bD_Q$ in (\ref{eqn:predual.check.1}) and using (\ref{eqn:norm.of.varphi_X,eta}) gives us the inequality `$\geq$' in (\ref{eqn:B.norms.H^infty(D_Q)}). The reverse inequality is trivial, 
therefore (\ref{eqn:B.norms.H^infty(D_Q)}) holds and we get that $\cX(\bD_Q)$ norms $H^\infty(\bD_Q)$.

\item Let us define
\begin{equation*}
\cC := \prod_{\varphi \in \cX(\bD_Q)} \|\varphi\| \overline{\bD} \subset \bC^{\cX(\bD_Q)},
\end{equation*}
where we endow $\bC^{\cX(\bD_Q)}$ with the topology of coordinate-wise convergence. By Tychonoff's theorem, $\cC$ is a compact set. Now, let $\Psi \colon H^\infty(\bD_Q) \to \bC^{\cX(\bD_Q)}$ be given by
\begin{equation*}
\Psi(F) := \left(\varphi(F)\right)_{\varphi \in \cX(\bD_Q)}, \foral F \in H^\infty(\bD_Q).
\end{equation*}
It is clear that $\Psi\left(\overline{B_1(H^\infty(\bD_Q))}\right) \subseteq \cC$. Therefore, it suffices to show that $\Psi$ restricts to a homeomorphism on $\overline{B_1(H^\infty(\bD_Q))}$ with respect to $\sigma := \sigma(H^\infty(\bD_Q),\cX(\bD_Q))$, and that $\Psi\left(\overline{B_1(H^\infty(\bD_Q))}\right)$ is closed in $\bC^{\cX(\bD_Q)}$.

\vspace{2mm}

\noindent \emph{\underline{Step 1: $\Psi$ is continuous on $H^\infty(\bD_Q)$ with respect to $\sigma$.}}

\vspace{2mm}

Let $\left\{ F_\kappa \right\}_{\kappa \in K} \subset H^\infty(\bD_Q)$ be a net that converges to some $F \in H^\infty(\bD_Q)$ under $\sigma$, i.e.,
\begin{equation} \label{eqn:sigma.conv.F_k}
\varphi(F_\kappa) \to \varphi(F), \foral \varphi \in \cX(\bD_Q).
\end{equation}
As $\bC^{\cX(\bD_Q)}$ is endowed with the topology of coordinate-wise convergence, it follows easily from (\ref{eqn:sigma.conv.F_k}) that $\Psi$ is continuous on $H^\infty(\bD_Q)$ with respect to $\sigma$.

\vspace{2mm}

\noindent \emph{\underline{Step 2: $\Psi$ is injective.}}

\vspace{2mm}

First, we introduce, for each $n \in \bN$ and $1 \leq j,k \leq n$, the functionals $\eta^{(n)}_{j,k} \in M_n^*$ given by
\begin{equation} \label{eqn:def.eta^n_j,k}
\eta^{(n)}_{j,k}(A) := \left(A\right)_{j,k}, \foral A \in M_n.
\end{equation}
where $\left(A\right)_{j,k}$ denotes the $(j,k)^\text{th}$ entry of $A$. Now, if $\Psi(F) = \Psi(G)$ for some $F, G \in H^\infty(\bD_Q)$ then, for each $n \in \bN$, we have
\begin{equation*}
\varphi_{X, \eta^{(n)}_{j,k}}(F) = \varphi_{X, \eta^{(n)}_{j,k}}(G), \foral X \in \bD_Q(n) \AND 1 \leq j,k \leq n,
\end{equation*}
where $\varphi_{X, \eta^{(n)}_{j,k}}$ are defined as in (\ref{eqn:canonical.element.predual}). It follows that $F = G$ and therefore, $\Psi$ is injective.

\vspace{2mm}

\noindent \emph{\underline{Step 3: $\Psi\left(\overline{B_1(H^\infty(\bD_Q))}\right)$ is closed in $\bC^{\cX(\bD_Q)}$.}}

\vspace{2mm}

Let $\left\{F_\kappa\right\}_{\kappa \in K} \subseteq \overline{B_1(H^\infty(\bD_Q))}$ be a net such that
\begin{equation*}
\lim_{\kappa \in K} \Psi(F_\kappa) = \left( c_\varphi \right)_{\varphi \in \cX(\bD_Q)} =: c \in \bC^{\cX(\bD_Q)}.
\end{equation*}
All we need to show is that there exists a nc function $F \colon \bM^d \to \bM^1$ such that
\begin{equation} \label{eqn:Psi(F)=c}
\Psi(F) = c,
\end{equation}
because it is easy to see that, in fact, we have $c \in \cC$, and that this implies using (\ref{eqn:Psi(F)=c}) and part (1) of the proof that $F \in \overline{B_1(H^\infty(\bD_Q))}$. To this end, consider $\eta^{(n)}_{j,k} \in M_n^*$ as in (\ref{eqn:def.eta^n_j,k}) for each $1 \leq j,k \leq n$ and $n \in \bN$, and define a graded map $F \colon \bM^d \to \bM^1$ via
\begin{equation} \label{eqn:def.F.for.Psi(F)=c}
F(X) := \left[ c_{\varphi_{X, \eta^{(n)}_{j,k}}} \right]_{n \times n}, \foral X \in \bD_Q(n) \AND n \in \bN.
\end{equation}
It is a routine calculation to check that $F$ is a nc map, so we only need to verify \eqref{eqn:Psi(F)=c}.

First, note that for any $n \in \bN$, (\ref{eqn:def.F.for.Psi(F)=c}) implies
\begin{equation} \label{eqn:F=c.entry-wise}
\varphi_{X, \eta^{(n)}_{j,k}}(F) = c_{\varphi_{X, \eta^{(n)}_{j,k}}}, \foral X \in \bD_Q(n) \AND 1 \leq j,k \leq n.
\end{equation}
Secondly, for any $\eta, \xi \in M_n^*$ and $a \in \bC$, one can check that
\begin{align} \label{eqn:linearity.in.eta}
c_{\varphi_{X, a \eta + \xi}} = a c_{\varphi{{X, \eta}}} + c_{\varphi_{X, \xi}}.
\end{align}
Lastly, note that for each $n \in \bN$ we have
\begin{equation} \label{eqn:M_n^*.is.span.eta^n_j,k}
M_n^* = \spn \left\{ \eta^{(n)}_{j,k} : 1 \leq j,k \leq n \right\}.
\end{equation}
Combining (\ref{eqn:F=c.entry-wise}), (\ref{eqn:linearity.in.eta}) and (\ref{eqn:M_n^*.is.span.eta^n_j,k}) at once shows that
\begin{equation*}
\varphi(F) = c_\varphi, \foral \varphi \in \cX(\bD_Q).
\end{equation*}
In particular, (\ref{eqn:Psi(F)=c}) holds, and we get that $\Psi\left(\overline{B_1(H^\infty(\bD_Q))}\right)$ is closed in $\bC^{\cX(\bD_Q)}$.

\vspace{2mm}

\noindent \emph{\underline{Step 4: $\Psi^{-1}$ is continuous on $\Psi\left(\overline{B_1(H^\infty(\bD_Q))}\right)$.}}

\vspace{2mm}

Let $K$ be an index set, and let $\left\{ F_\kappa \right\}_{\kappa \in K} \subset \overline{B_1(H^\infty(\bD_Q))}$ be a net such that $\left( \varphi(F_\kappa) \right)_{\varphi \in \cX(\bD_Q)}$ is convergent in $\bC^{\cX(\bD_Q)}$. From \emph{Step 3}, we know that
\begin{equation*}
\left(\varphi(F_\kappa)\right)_{\varphi \in \cX(\bD_Q)} \to \left(\varphi(F)\right)_{\varphi \in \cX(\bD_Q)}, \text{ for some } F \in \overline{B_1(H^\infty(\bD_Q))}.
\end{equation*}
Since $\bC^{\cX(\bD_Q)}$ is endowed with the topology of coordinate-wise convergence, it follows that
\begin{equation*}
F_\kappa \xrightarrow{\sigma} F.
\end{equation*}
Thus, $\Psi^{-1}$ is continuous on $\Psi\left(\overline{B_1(H^\infty(\bD_Q))}\right)$, and we conclude that (2) holds for $\cX(\bD_Q)$.

\end{enumerate}

It follows from $(1) \AND (2)$ that $\cX(\bD_Q)$ is a predual of $H^\infty(\bD_Q)$. We now show that $\cX(\bD_Q)$ is the unique predual of $H^\infty(\bD_Q)$ for which $\Phi_X$ (as in \eqref{eqn:def.Phi_X}) is $\sigma$ continuous for each $X \in \bD_Q$. To this end, fix $X \in \bD_Q(n)$ for some $n \in \bN$, and let $\left\{F_\kappa\right\}_{\kappa \in K}$ be a net in $H^\infty(\bD_Q)$ that converges to some $F \in H^\infty(\bD_Q)$ with respect to $\sigma$, i.e.,
\begin{equation*}
\lim_{\kappa \in K} \varphi(F_\kappa) = \varphi(F), \foral \varphi \in \cX(\bD_Q).
\end{equation*}
Using $\varphi_{X, \eta^{(n)}_{j,k}} \in \cX(\bD_Q)$ as in \emph{Step 2}, we know that $F_\kappa(X)$ converges to $F(X)$ entry-wise, i.e,
\begin{equation} \label{eqn:entry-wise.conv}
\lim_{\kappa \in K}(F_\kappa(X))_{j,k} = (F(X))_{j,k}, \foral 1 \leq j,k \leq n.
\end{equation}
In particular, $F_\kappa(X)$ converges to $F(X)$ with respect to $\left\|\cdot\right\|$. As $n \in \bN$ and $X \in \bD_Q(n)$ were arbitrarily chosen, it follows that $\Phi_X$ is $\sigma$ continuous for all $X \in \bD_Q$.

Let $\widetilde{\cX} \subseteq (H^\infty(\bD_Q))^*$ be another predual for which the maps $\Phi_X$ are $\widetilde{\sigma} := \sigma(H^\infty(\bD_Q), \widetilde{\cX})$ continuous. Now, for any given $n \in \bN$, the functional $\varphi_{X, \eta}$ (as in \eqref{eqn:canonical.element.predual}) is $\widetilde{\sigma}$ continuous for each $X \in \bD_Q(n)$ and $\eta \in M_n^*$. A standard exercise in functional analysis then shows that all such $\varphi_{X, \eta}$ must lie in $\widetilde{\cX}$. In particular, we get $\cX(\bD_Q) \subseteq \widetilde{\cX}$. If $0 \neq \widetilde{\varphi} \in \widetilde{\cX} \setminus \cX(\bD_Q)$, then by the Hahn-Banach extension theorem there exists some $F \in H^\infty(\bD_Q)$ such that
\begin{equation} \label{eqn:Hahn-Banach.cor}
\widetilde{\varphi}(F) = 1 \AND \varphi(F) = 0, \foral \varphi \in \cX(\bD_Q).
\end{equation}
However, the second half of \eqref{eqn:Hahn-Banach.cor} implies that for every $n \in \bN$ and $X \in \bD_Q(n)$ we have
\begin{equation*}
0 = \varphi_{X, \eta^{(n)}_{j,k}}(F) = (F(X))_{j,k}, \foral 1 \leq j,k \leq n.
\end{equation*}
This means that $F = 0$, and that $\widetilde{\varphi}(F) = 0$, which is absurd considering \eqref{eqn:Hahn-Banach.cor}. Therefore, $\widetilde{\cX} = \cX(\bD_Q)$, and it is the unique predual such that $\Phi_X$ is $\sigma$ continuous for every $X \in \bD_Q$.
\end{proof}

\begin{remark} \label{remark:weak-*.topology}
In view of Theorem \ref{theorem:predual.H^infty(D_Q)}, we denote $(H^\infty(\bD_Q))_* := \cX(\bD_Q)$ to be the unique predual of $H^\infty(\bD_Q)$ such that the point evaluation $\Phi_X$ is weak-* continuous for all $X \in \bD_Q$. 
From now on we shall refer to the topology $\sigma(H^\infty(\bD_Q), \cX(\bD_Q))$ as the (canonical) weak-* topology on $H^\infty(\bD_Q)$. 
It is known that the predual of $H^\infty(\fB_d)$ is (strongly) unique (see \cite{Ando78} for the case $d=1$ and \cite{DW11} for the case $d \geq 2$), but we do not know if this holds in general. 
\end{remark}

\begin{remark} \label{remark:equivalence.weak-*.bdd.pt-wise}
For $\bD_Q = \fB_d$, it is well-known that $H^\infty(\fB_d)$ can be represented as the \emph{multiplier algebra} of the \emph{nc Drury--Arveson space} $\cH^2_d$ and, therefore, it has a natural weak operator topology. Davidson and Pitts worked out that the weak operator topology and the weak-* topology on $H^\infty(\bD_Q)$ coincide \cite[Corollary 2.12]{DP99}. In \cite[Section 2]{SampatShalit25}, we showed that such a multiplier algebra representation might not be possible for the algebra of bounded nc functions on a general $\bD_Q$.
\end{remark}

It is straightforward to conclude that the weak-* topology and the topology of pointwise convergence on $H^\infty(\bD_Q)$ coincide. As a result, the observations from \cite[Sections 3 and 4]{SampatShalit25} can now be understood in terms of the canonical weak-* topology. We recall these results for the reader's convenience. 
Note that, strictly speaking, the topology of {\em bounded} pointwise convergence used in \cite{SampatShalit25} does not coincide with the weak-* topology, but continuity of linear maps relative to one is equivalent to continuity relative to the other (see \cite[Theorem V.5.6]{dunford1988linear}), and similarly for closures. 
Therefore, we may transfer the results to here. 

\begin{proposition} \label{prop:homogeneous.exp.properties}
Let $F = \sum_k F_k$ be the homogeneous expansion of any given $F \in H^\infty(\bD_Q)$ (as in \eqref{eqn:kth.homogeneous.term}).

\begin{enumerate}[itemindent=*, leftmargin=0mm, itemsep=2mm]

\item The map $P_k \colon H^\infty(\bD_Q) \to \bC \langle Z \rangle$ given by
\begin{equation*}
P_k(F) = F_k, \foral F \in H^\infty(\bD_Q)
\end{equation*}
is weak-* continuous and completely contractive.

\item For any $F \in H^\infty(\bD_Q)$, the weak-* limit of Ces\`aro sums $\Sigma_k(F)$ of the homogeneous expansion $\sum_k F_k$ exists, and it is equal to $F$.

\end{enumerate}
\end{proposition}

\begin{theorem}[Homogeneous NC Nullstellensatz] \label{theorem:hom.Nullstellensatz}
A weak-* closed ideal $\cJ \triangleleft H^\infty(\bD_Q)$ is homogeneous if and only if $\cJ = \overline{\cI}^{w^*}$ for some homogeneous ideal $\cI \triangleleft \bC \langle Z \rangle$.

Moreover, we have the following nc Nullstellensatz for homogeneous ideals $\cJ \triangleleft H^\infty(\bD_Q)$
\begin{equation*}
I\left(V_{\bD_Q}(\cJ)\right) = \overline{\cJ}^{w^*}.
\end{equation*}
\end{theorem}

\begin{theorem} \label{theorem:H(V).as.restr.algebra}
Let $\fV \subseteq \bD_Q$ be a nc subvariety. The map $R \colon H^\infty(\bD_Q) \to H^\infty(\fV)$ given by
\begin{equation}
R(F) = F \big\vert_{\fV}, \foral F \in H^\infty(\bD_Q)
\end{equation}
is a complete contraction with $\ker(R) = I(\fV)$.

Furthermore, the induced quotient map $\overline{R} \colon H^\infty(\bD_Q)/I(\fV) \to H^\infty(\fV)$ is a completely isometric isomorphism.
\end{theorem}

It is easy to check that for a nc subvariety $\fV \subseteq \bD_Q$ we have $I(\fV) = (\cX(\fV))^\perp$, i.e., the annihilator of the subspace $\cX(\fV) \subseteq (H^\infty(\bD_Q))_*$ given by
\begin{equation*}
\cX(\fV) := \spn{\left\{ \varphi_{X, \eta} : X \in \fV(n) \text{ and } \eta \in M_n^*, \foral n \in \bN \right\}},
\end{equation*}
where $\varphi_{X,\eta}$ is as in \eqref{eqn:canonical.element.predual}. In particular, 
\[
\cX(\fV)^* \cong H^\infty(\bD_Q)/(\cX(\fV))^\perp = H^\infty(\bD_Q)/I(\fV) \cong H^\infty(\fV), 
\]
and this shows that $H^\infty(\fV)$ is also a dual space. 
It can be checked that $(H^\infty(\fV))_* = \cX(\fV)$ is the unique such predual for which point evaluations are weak-* continuous.

\begin{remark} \label{remark:predual.H(V)}
The keen reader may notice that the argument of Theorem \ref{theorem:predual.H^infty(D_Q)} works for any nc set in place of $\bD_Q$. However, we would have liked to highlight this connection between the predual of $H^\infty(\fV)$ with a subspace of $(H^\infty(\bD_Q))_*$ anyway.
\end{remark}

The weak-* topology on $H^\infty(\fV)$ can be characterized by the property that $f_\kappa \xrightarrow{w^*} f$ in $H^\infty(\fV)$ if and only if $f_\kappa(X) \to f(X)$ for all $X \in \fV$. We leave the proofs of this fact and the following corollary to the reader.

\begin{corollary} \label{corollary:R.is.w*.cont}
The maps $R$ and $\overline{R}^{-1}$ as in Theorem \ref{theorem:H(V).as.restr.algebra} are weak-* continuous.
\end{corollary}


\section{Classification of the algebras \texorpdfstring{$H^\infty(\fV)$}{H(V)}}\label{sec:class.alg.H(V)}

Let $\fV \subseteq \bD_Q$ be a subvariety of some nc operator ball. In this section, we classify $H^\infty(\fV)$ up to weak-* continuous and completely isometric isomorphisms. We start by studying the finite dimensional representations of $H^\infty(\fV)$.

\subsection{Finite dimensional representations}\label{subsec:fin.dim.reps}

Let $\Rep_n(H^\infty(\fV))$ be the space of all bounded $n$-dimensional representations of $H^\infty(\fV)$ and $\Rep^{cc}_n(H^\infty(\fV))$ be the space of all completely contractive $n$-dimensional representations of $H^\infty(\fV)$. We write
\begin{equation*}
\Rep(H^\infty(\fV)) = \bigsqcup_{n = 1}^\infty \Rep_n(H^\infty(\fV)),
\end{equation*}
and similarly
\begin{equation*}
\Rep^{cc}(H^\infty(\fV)) = \bigsqcup_{n = 1}^\infty \Rep^{cc}_n(H^\infty(\fV)).
\end{equation*}

We have a natural projection $\pi \colon \Rep^{cc}(H^\infty(\fV)) \to \overline{\bD_Q}$ given by
\begin{equation*}
\pi(\Phi) = (\Phi(Z_1 \vert_\fV),\dots,\Phi(Z_d \vert_\fV)) \foral \Phi \in \Rep^{cc}(H^\infty(\fV)).
\end{equation*}
That $\pi(\Phi)$ lies in $\overline{\bD_Q}$ was shown in \cite[Section 5]{SampatShalit25}. It is also clear from Theorem \ref{theorem:H(V).as.restr.algebra} and Corollary \ref{corollary:R.is.w*.cont} that
\begin{equation}\label{eqn:rep.on.V.in.terms.of.rep.on.D_Q}
\Rep(H^\infty(\fV)) \cong \left\{ \Phi \in \Rep(H^\infty(\bD_Q)) : I(\fV) \subseteq \ker \Phi \right\},
\end{equation}
and that \eqref{eqn:rep.on.V.in.terms.of.rep.on.D_Q} also holds for the corresponding classes of weak-* continuous and/or completely contractive representations.

Given any $n \in \bN$ and $X \in \bD_Q(n)$, recall from \eqref{eqn:def.Phi_X} the point evaluation map $\Phi_X \in \Rep^{cc}_n (H^\infty(\bD_Q))$ given by $\Phi_X \colon F \mapsto F(X)$. Note that $\pi(\Phi_X) = X$ in this case and, conversely, using Proposition \ref{prop:homogeneous.exp.properties} $(2)$, it is also clear that $\Phi_X$ is the unique weak-* continuous representation in $\pi^{-1}(X)$.
This leads us to the following result from \cite[Section 5]{SampatShalit25}, which is crucial in this discussion so we mention it here.

\begin{theorem}\label{thm:projn.of.reps.on.V}
$\pi$ is a continuous map that restricts to a homeomorphism between the weak-* continuous $n$-dimensional representations in $\pi^{-1}(\bD_Q)$ and $\fV(n)$.
\end{theorem}

For $\fV = \bD_Q = \fB_d$, it is known \cite[Theorem 3.2]{DP98b} that
\begin{equation*}
	\pi^{-1}(X) = \{\Phi_X\} \foral X \in \fB_d.
\end{equation*}
We are unable to prove this for a general $\bD_Q$ in place of the row ball $\fB_d$. 
We discuss the challenges in generalizing this result and also mention some alternatives to tackle this problem for a general $\bD_Q$ in Section \ref{sec:rep_ball}. Let $\Rep^{cc,w^*}(H^\infty(\fV))$ be the nc set consisting of weak-* continuous elements of $\Rep^{cc}(H^\infty(\fV))$ and denote
\begin{equation*}
\overline{\fV}^p = \pi(\Rep^{cc,w^*}(H^\infty(\fV))) \subset \overline{\bD_Q}.
\end{equation*}

\begin{lemma}\label{lemma:uniq.reps.over.V^p}
Given $X \in \overline{\fV}^p$, there is a unique $\Phi \in \Rep^{cc,w^*}(H^\infty(\fV))$ with $\pi(\Phi) = X$, namely $\Phi_X$.
\end{lemma}

\begin{proof}
If $\Phi \in \pi^{-1}(X)$, then note that
\begin{equation}\label{eqn:rep.act.on.poly}
\Phi(P) = P(X) \foral P \in \bC\langle Z \rangle.
\end{equation}
Thus, it suffices to show that $\bC \langle Z \rangle$ is weak-* dense in $H^\infty(\fV)$, since it would force the action of $\Phi$ via \eqref{eqn:rep.act.on.poly}. However, this is clear from Proposition \ref{prop:homogeneous.exp.properties} $(2)$ and the fact that any $f \in H^\infty(\fV)$ has a norm preserving extension $F \in H^\infty(\bD_Q)$ \cite[Corollary 3.4]{BMV18}.
\end{proof}

\begin{remark}\label{rem:pure.tuples}
It is known \cite[Remark 6.2]{SSS18} that $\overline{\fB_d}^p$ is the set of \emph{pure tuples} $X$, i.e.,
\begin{equation*}
\sotlim_{k \to \infty} \sum_{|\alpha| = k} (X^\alpha)^* = 0.
\end{equation*}
For a general $\bD_Q$, we cannot say much. 
However, we note that there are always points in $\overline{\bD_Q}^p$ living over the boundary $\partial \bD_Q$. Indeed, every \emph{$k$-nilpotent tuple} $X \in \overline{\bD_Q}$ lies in $\overline{\bD_Q}^p$ (recall that $X$ is said to be $k$-nilpotent if $X^\alpha = 0$ for all $|\alpha| \geq k$).
\end{remark}

We need the following useful description of $\overline{\fV}^p$ for arbitrary $\fV \subseteq \bD_Q$, which is obtained as in the discussion surrounding \eqref{eqn:rep.on.V.in.terms.of.rep.on.D_Q}.

\begin{lemma}\label{lemma:V^p.in.terms.of.D_Q^p}
For any subvariety $\fV \subseteq \bD_Q$, we can identify $\overline{\fV}^p$ as
\begin{equation}\label{eqn:V^p.in.terms.of.D_Q^p}
\overline{\fV}^p = \left\{ X \in \overline{\bD_Q}^p : I(\fV) \subseteq \ker \Phi_X \right\}.
\end{equation}
\end{lemma}



We also record the following important property of $\partial \ol{\fV}^p := \ol{\fV}^p \cap \partial \bD_Q$. 

\begin{corollary}\label{cor:no.scalar.point}
For any subvariety $\fV$ of a given nc operator ball $\bD_Q$,
\[
\partial \overline{\fV}^p(1)= \emptyset.  
\]
\end{corollary}
\begin{proof}
Thanks to Lemma \ref{lemma:V^p.in.terms.of.D_Q^p}, it suffices to prove the claim for $\fV = \bD_Q$. 
Now, $B = \bD_Q(1)$ is a bounded convex symmetric domain in $\bC^d$. 
Thus, for every $x \in \partial B$, there exists a linear functional $\phi$ on $\bC^d$ such that $\phi(x) = 1$ and $|\phi(z)| < 1$ for all $z \in B$. 
For each $X = (x_{j,k}) \in \bD_Q(n)$, the matrix $(\phi(x_{j,k}))\in M_n$ is a strict contraction (see the proof of Proposition \ref{prop:boundedness.D_Q}). 
Therefore, the map $F \mapsto F \circ \phi$ is an injective, completely contractive and weak-* continuous homomorphism of $H^\infty(\bD) = H^\infty(\fD_1)$ into $H^\infty(\bD_Q)$. 
Now, if $x \in \ol{\bD_Q}^p$ then the character of point evaluation at $x$ is weak-* continuous, and this gives rise rise to a weak-* continuous point evaluation evaluation functional of $H^\infty(\bD)$ at the point $1$, which is impossible. 
\end{proof}

We next see how finite dimensional representations help us classify these algebras.

\subsection{Basic classification results}\label{subsec:basic.class.res}

For $i = 1,2$, let $\fV_i \subseteq \bD_{Q_i}$ be subvarieties of some nc operator balls. Then, every bounded homomorphism $\varphi \colon H^\infty(\fV_1) \to H^\infty(\fV_2)$ induces a map $\varphi^* \colon \Rep(H^\infty(\fV_2)) \to \Rep(H^\infty(\fV_1))$ via
\begin{equation}\label{eqn:map.between.rep.induced.by.hom}
\varphi^*(\Phi) = \Phi \circ \varphi \foral \Phi \in \Rep(H^\infty(\fV_2)).
\end{equation}
It is clear that whenever $\varphi$ is weak-* continuous and/or completely contractive, then $\varphi^*$ preserves weak-* continuous and/or completely contractive representations. This leads us to the following basic observation.

\begin{theorem}\label{thm:weak-*.cc.map.implies.nc.map}
For $i = 1,2$, let $\fV_i \subseteq \bD_{Q_i}$ be subvarieties of some nc operator balls. If $\varphi \colon H^\infty(\fV_1) \to H^\infty(\fV_2)$ is a weak-* continuous completely contractive homomorphism, then there is a nc map $F \colon \overline{\fV_2}^p \to \overline{\fV_1}^p$ such that
\begin{equation}\label{eqn:hom.induced.by.nc.map}
\varphi(f) = f \circ F \foral f \in H^\infty(\fV_1).
\end{equation}

Conversely, every nc map $F \colon \fV_2 \to \fV_1$ induces a weak-* continuous completely contractive homomorphism $\varphi \colon H^\infty(\fV_1) \to H^\infty(\fV_2)$.
\end{theorem}

\begin{proof}
Let $\varphi$ be as in the hypothesis and define $F \colon \overline{\fV_2}^p \to \overline{\fV_1}^p$ via
\begin{equation*}
F(X) = \pi(\varphi^*(\Phi_X)) \foral X \in \overline{\fV_2}^p.
\end{equation*}
Here, $\varphi^*$ is as in \eqref{eqn:map.between.rep.induced.by.hom} and $\Phi_X$ is as in \eqref{eqn:def.Phi_X}. 
Since $\varphi$ is weak-* continuous, $\varphi^*$ maps weak-* continuous representations to weak-* continuous representations. 
Thus, $\varphi^*(\Phi_X) = \Phi_{F(X)}$, because, by Lemma \ref{lemma:uniq.reps.over.V^p}, the representation $\Phi_{F(X)}$ is the unique weak-* continuous element in $\pi^{-1}(F(X))$. 
For any $f \in H^\infty(\fV_1)$ and $X \in \fV_1$, we get
\begin{equation*}
\varphi(f)(X) = (\Phi_X \circ \varphi)(f) = (\varphi^*(\Phi_X))(f) = \Phi_{F(X)}(f) = (f \circ F)(X),
\end{equation*}
and we therefore obtain \eqref{eqn:hom.induced.by.nc.map}.

The converse is clear, since we can define $\varphi$ via \eqref{eqn:hom.induced.by.nc.map}, which is clearly pointwise, hence weak-* continuous and completely contractive.
\end{proof}

We obtain the following basic classification theorem, which is an easy corollary of Theorem \ref{thm:weak-*.cc.map.implies.nc.map}. Recall that a \emph{nc biholomorphism} is a bijective nc holomorphic map between two nc sets.

\begin{corollary}\label{cor:basic.class.thm}
For $i = 1,2$, let $\fV_i \subseteq \bD_{Q_i}$ be subvarieties of some nc operator balls. If $\varphi \colon H^\infty(\fV_1) \to H^\infty(\fV_2)$ is a weak-* continuous completely isometric isomorphism, then there is a nc biholomorphism $F \colon \overline{\fV_2}^p \to \overline{\fV_1}^p$.

On the other hand, if there is a nc biholomorphism $F \colon \fV_2 \to \fV_1$, then there is a weak-* continuous completely isometric isomorphism $\varphi \colon H^\infty(\fV_1) \to H^\infty(\fV_2)$.

In either of these cases, the maps $\varphi$ and $F$ are related by \eqref{eqn:hom.induced.by.nc.map}.
\end{corollary}

Unlike the varieties $\fV_i$, the nc sets $\overline{\fV_i}^p$ are not natural geometric objects, and their intrinsic interest is questionable. 
Thus, we would like to strengthen the conclusion of the first part of the above theorem by showing that the biholomorphism $F \colon \overline{\fV_2}^p \to \overline{\fV_1}^p$ restricts to a biholomorphism of $\fV_2$ onto $\fV_1$. 
For this, we shall require a following generalization of the maximum modulus principle to the nc setting, which we shall refer to as {\em the boundary value principle}. 
This is a strengthening and unification of several similar results from an earlier work of ours (see Theorem 3.2, Remark 6.2 and Theorem 6.5 in \cite{SampatShalit25}). 

In the next lemma and theorem, we shall use the following notation. 
If $\cE$ is a an operator space (not necessarily finite dimensional), we shall write $\bB_\cE$ for the open nc unit ball
\[
\bB_\cE = \bigsqcup_{n=1}^\infty \bB_\cE(n) = \bigsqcup_{n=1}^\infty \{X \in M_n(\cE) : \|X\| < 1\}.
\]
We also define its closure $\ol{\bB}_\cE = \sqcup_{n=1}^\infty \ol{\bB_\cE(n)}$ and boundary $\partial \bB_\cE = \sqcup_{n=1}^\infty \partial \bB_\cE(n)$, where
\[
 \partial \bB_\cE(n) = \{X \in M_n(\cE) : \|X\| = 1\}.
\]

\begin{lemma}\label{lemma:BVP.preliminary}
Let $\Omega \subset \bM^d$ be a nc domain and $\cE$ be an operator space. Suppose $G \colon \Omega \to \ol{\bB}_\cE$ is a nc holomorphic map. If $G(X_0) \in \partial \bB_\cE$ for some $X_0 \in \Omega$, then $G(\Omega) \subseteq \partial\bB_\cE$.
\end{lemma}
\begin{proof}
If $G(X_0) \in \partial \bB_\cE(n)$ then $\|G(X_0)\| = 1$. 
Let $\Lambda \in M_n(\cE)^*$ be a functional of norm $1$ such that $\Lambda(G(X_0)) = 1$, and consider the holomorphic function $g = \Lambda \circ G \colon \Omega(n) \to \bC$. 
By the maximum modulus principle for scalar-valued multivariate holomorphic functions, $g$ is constant on the domain $\Omega(n)$. 
Since $\|\Lambda\| = 1$, it follows that $\|G(X)\|=1$ for all $X \in \Omega(n)$. 

We have shown that if $X_0 \in \Omega(n)$ and $G(X_0) \in \partial \bB_\cE(n)$, then $G(\Omega(n))\subseteq \partial \bB_\cE(n)$. Now a standard nc trick yields $G(\Omega(m))\subseteq \partial \bB_\cE(m)$ for all $m$, as follows. We need to show that $\|G(X)\| = 1$ for every $m$ and every $X \in \Omega(m)$. Let us write $X_0^{(m)} = \underbrace{X_0 \oplus \cdots \oplus X_0}_{m \textrm{ times}} \in \Omega(mn)$, and note that
\begin{equation*}
    \|G(X_0^{(m)})\| = \left\| \oplus_{k=1}^m G(X_0)\right\| = \| G(X_0) \| = 1.
\end{equation*}
Then, by the first paragraph of the proof, $G(\Omega(mn)) \subseteq \partial \bB_\cE$ and so $\left\|G(X^{(n)})\right\| = 1$. It follows that $\|G(X)\| = 1$ and we are done.
\end{proof}

\begin{theorem}[The boundary value principle]\label{thm:bdy.val.prin}
Let $\bD_{Q}$ and $\bD_{P}$ be nc operator balls and let $\fV \subseteq \bD_{Q}$ be a subvariety. For any nc map $F \colon \fV \to \overline{\bD_{P}}$, we have the following dichotomy:

\begin{enumerate}[leftmargin=*]
\item If $F(X_0) \in \bD_{P}$ for some $X_0 \in \fV$, then $F(\fV) \subset \bD_{P}$.

\item If $F(X_0) \in \partial \bD_{P}$ for some $X_0 \in \fV$, then $F(\fV) \subset \partial \bD_{Q}$.
\end{enumerate}
\end{theorem}

\begin{proof}
Let $P(Z) = \sum_{j=1}^{d_2} P_j Z_j \colon \bM^{d_2} \to \cL(\cU,\cV)_{nc}$ for some Hilbert spaces $\cU, \cV$ with
\[
\cE_2 := \operatorname{span}\left\{P_1, \ldots, P_{d_2}\right\} \subseteq \cL(\cU,\cV).
\]
Define $G_0 \colon \fV \rightarrow \cL(\cU,\cV)_{nc}$ by $G_0 = P \circ F$. 
Then $G_0$ is a nc map such that $\| G_0(X) \| \leq 1$ for all $X \in \fV$. 
As a nc variety, $\fV$ is a \emph{relatively full} nc subset of $\bD_{Q}$, so by the Ball-Marx-Vinnikov extension theorem \cite[Corollary 3.4]{BMV18}, there exists an extension of $G_0$ to a nc map $G \colon \bD_{Q} \to  \cL(\cU,\cV)_{nc}$ such that 
\[
\sup_{X \in \bD_{Q}} \| G(X) \| = \sup_{X \in \fV} \| G_0(X) \| \leq 1. 
\] 
Clearly, $G$ satisfies the assumptions of Lemma \ref{lemma:BVP.preliminary} with $\Omega = \bD_{Q}$ and $\cE = \cL(\cU,\cV)$. 
Therefore, if $F(X_0) \in \partial \bD_{P}$ for $X_0 \in \fV \subseteq \Omega$, then $G(X_0) \in \partial \bB_\cE$, and by the lemma $\|G(X)\| = 1$ for all $X \in \Omega$. 
In particular, $\|G(X)\| = 1$, which is the same as $F(X) \in \partial \bD_{P}$, for all $X \in \fV$.
\end{proof}

We now arrive at our central general classification result. 
Let us say that a nc variety $\fV$ {\em contains a scalar point} if $\fV(1) \neq \emptyset$. 
The class of varieties that contain a scalar point is rather broad, including, in particular, all homogeneous varieties. 
However, there are varieties with no scalar points (see \cite[Example 4.4]{Sha18}), and we do not know how to generalize the following theorem to that case.

\begin{theorem}\label{thm:gen.class.thm}
For $i = 1,2$, let $\fV_i \subseteq \bD_{Q_i}$ be subvarieties of some nc operator balls. 
If $\fV_2$ contains a scalar point, then the following are equivalent:
\begin{enumerate}[leftmargin=*]
\item There is a weak-* continuous completely isometric isomorphism $\varphi \colon H^\infty(\fV_1) \to H^\infty(\fV_2)$.
\item There is a nc biholomorphism $F \colon \overline{\fV_2}^p \to \overline{\fV_1}^p$.
\item There is a nc biholomorphism $F \colon \fV_2 \to \fV_1$.
\end{enumerate}

Moreover, every weak-* continuous and completely isometric isomorphism $\varphi$ is induced by a nc biholomorphism $F$ as in \eqref{eqn:hom.induced.by.nc.map}.
\end{theorem}
\begin{proof}
$(1) \Rightarrow (2)$ and $(3) \Rightarrow (1)$ are the conclusions of Corollary \ref{cor:basic.class.thm}. To show $(2) \Rightarrow (3)$, all we need to show is that every nc biholomorphism $F \colon \overline{\fV_2}^p \to \overline{\fV_1}^p$ maps $\fV_2$ into $\fV_1$. 

To this end, suppose $F \colon \overline{\fV_2}^p \to \overline{\fV_1}^p$ is a nc biholomorphism with $F(X_0) \in \partial \bD_{Q_1}$ for some $X_0 \in \fV_2$. The boundary value principle then shows that $F(\fV_2) \subset \partial \bD_{Q_1}$. But then this also holds at the first level, so the scalar points in $\fV_2$ are mapped into $\partial \overline{\fV_1}^p$, which by Corollary \ref{cor:no.scalar.point} is empty.
It follows therefore that $F(\fV_2) \subseteq \fV_1$ must hold. 
Noting that now we know that $\fV_1$ also contains a scalar point, the proof is complete by symmetry. 
\end{proof}

\subsection{Classification results for injective nc operator balls}\label{subsec:inj.nc.op.balls}

Fix $Q \colon \bM^d \to \cL(\cR,\cS)$ as in \eqref{eqn:op.val.lin.nc.poly}. We say that the corresponding nc operator ball $\bD_Q$ is \emph{injective} if the operator space
\begin{equation*}
\cE := \spn \{ Q_1, \dots, Q_d \} \subset \cL(\cR,\cS)
\end{equation*}
is injective in the sense of \cite{Rua89}. In particular, the injectivity of $\bD_Q$ is equivalent to the existence of a completely contractive projection $\Pi \colon \cL(\cR,\cS) \to \cE$. It is easy to see that both $\fB_d$ and $\fD_d$ are injective. Moreover, $\bD_Q \subset \bM^4$ corresponding to
\begin{equation*}
Q(Z) = \begin{bmatrix}
Z_1 & Z_2 \\
Z_4 & Z_3
\end{bmatrix}
\end{equation*}
is injective, but $\bD_Q \subset \bM^3$ corresponding to
\begin{equation*}
Q(Z) = \begin{bmatrix}
Z_1 & Z_2 \\
0 & Z_3
\end{bmatrix}
\end{equation*}
is not injective \cite[Examples 6.4 and 6.17]{SampatShalit25}. 

In \cite[Theorem 6.5]{SampatShalit25} we proved that every nc map $F_0 \colon \fX \to \ol{\bD_{Q_2}}$ from a relatively full nc subset $\fX$ of a nc operator ball $\bD_{Q_1}$ into the closure of an injective nc operator ball can be extended to a nc map $F \colon \bD_{Q_1} \to \ol{\bD_{Q_2}}$ (the assumption that $\bD_{Q_2}$ is injective cannot be dropped, in general). 
Combining this extension theorem with Theorem \ref{thm:gen.class.thm} we obtain a stronger classification result.

\begin{theorem}\label{thm:class.thm.inj.balls.and.hom.var}
For $i = 1,2$, let $\fV_i \subseteq \bD_{Q_i}$ be subvarieties of two injective nc operator balls. If $\fV_2$ contains a scalar point, then the following are equivalent:

\begin{enumerate}[leftmargin=*]
\item There is a weak-* continuous completely isometric isomorphism $\varphi \colon H^\infty(\fV_1) \to H^\infty(\fV_2)$.

\item There are nc maps $G \colon \bD_{Q_1} \to \bD_{Q_2}$ and $F \colon \bD_{Q_2} \to \bD_{Q_1}$ such that $G \vert_{\fV_1}$ is a bijection of $\fV_1$ onto $\fV_2$ and such that $G \vert_{\fV_1} = (F \vert_{\fV_2})^{-1}$.
\end{enumerate}
\end{theorem}

It is now natural to ask if one can prove a stronger version of Theorem \ref{thm:class.thm.inj.balls.and.hom.var} in which the maps $G$ and $F$ in condition $(2)$ can be chosen to be inverses of each other on the nc operator balls. 
There is a simple geometric obstruction for such a result to hold.  
For instance, consider two copies $\fV_1$ and $\fV_2$ of $\fD_1$, the first considered as a subvariety in itself and the second sitting inside $\fD_2$ as the subvariety $\{(X,0) : X \in \fD_1\}$. 
Then, clearly, $H^\infty(\fV_1) \cong H^\infty(\fV_2)$ and the obvious map between $\fV_1$ and $\fV_2$ is a nc biholomorphism that induces this isomorphism but cannot be extended to a nc biholomorphism between the ambient balls $\fD_1$ and $\fD_2$.  

It turns out that if we assume that the varieties are homogeneous and the balls are injective, then we can obtain the stronger result under a reasonable minimality assumption that was introduced in \cite{SSS18} for the row ball, and was later shown to also work for any injective nc operator ball in \cite{SampatShalit25}. While this condition was introduced differently in both these references, we now understand that this condition can be simply understood with the following more convenient equivalent definition given by Shamovich \cite[Lemma 3.4]{Sha18} (see also \cite[Theorem 6.14]{SampatShalit25}).

\begin{definition}\label{def:matrix-span.variety}
A subvariety $\fV \subseteq \bD_Q$ of some nc operator ball is said to be \emph{matrix-spanning} if $I(\fV)$ contains no linear homogeneous polynomial.
\end{definition}

We can now combine our results from \cite{SampatShalit25} with Theorem \ref{thm:class.thm.inj.balls.and.hom.var} to obtain our sharpest classification theorem.

\begin{theorem}\label{thm:sharpest.classification.thm}
For $i = 1,2$, let $\fV_i \subseteq \bD_{Q_i}$ be matrix-spanning homogeneous subvarieties of some injective nc operator balls. Then, the following are equivalent.

\begin{enumerate}[leftmargin=*]
\item\label{it:1} There is a weak-* continuous completely isometric isomorphism $\varphi \colon H^\infty(\fV_1) \to H^\infty(\fV_2)$.

\item\label{it:2} There is a nc biholomorphism of $\fV_2$ onto $\fV_1$.

\item\label{it:3} There is a nc biholomorphism $F \colon \bD_{Q_2} \to \bD_{Q_1}$ such that $F(\fV_2) = \fV_1$.

\item\label{it:4}  There is a linear isomorphism $L \colon \bD_{Q_2} \to \bD_{Q_1}$ such that $L(\fV_2) = \fV_1$.

\item\label{it:5} There is a completely isometric isomorphism $\widetilde{\varphi} \colon A(\fV_1) \to A(\fV_2)$.
\end{enumerate}

Moreover, every isomorphism $\varphi$ as above is induced by composition with such a biholomorphism $F$ as in \eqref{eqn:hom.induced.by.nc.map}.
\end{theorem}

\begin{proof}
By \cite[Theorem 6.21]{SampatShalit25}, conditions (\ref{it:3}), (\ref{it:4}), and (\ref{it:5}) are equivalent. 
By Theorem \ref{thm:class.thm.inj.balls.and.hom.var}, conditions (\ref{it:1}) and (\ref{it:2}) are equivalent (homogeneous varieties contain the scalar point $0$). 
Now, (\ref{it:3}) clearly implies (\ref{it:2}). 
On the other hand, by \cite[Theorem 6.12]{SampatShalit25}, a biholomorphism between two matrix spanning homogeneous subvarieties of injective operator balls extends to a biholomorphism between the balls, thus (\ref{it:2}) implies (\ref{it:3}). 
Altogether, we have shown that all conditions are equivalent. 
\end{proof}

The condition that two nc domains are biholomorphic is typically very rigid. 
We take the opportunity to refer the reader to the deep paper \cite{AHKM} (see also the references therein) on the possible biholomorphic maps between free spectrahedra. 

One may wonder whether it is necessary to assume in the above theorem that the varieties are homogeneous. 
Following earlier work in \cite{Sha18}, Belinschi and Shamovich obtained a variant of the above theorem in the case of the row ball, i.e., $\bD_{Q_1} = \bD_{Q_2} = \fB_d$ with no assumption on the varieties \cite[Theorem 5.7]{BS+}.
However, their methods of proof rely on the geometry of $\fB_d$ in an essential way. 
And indeed, the following example shows that for general operator balls the theorem might fail if the varieties are not both homogeneous. 

\begin{example}\label{ex:cant.drop.homo}
Consider the following two varieties in $\fD_2$:
\[
\fV_1 = \{(X,X^2) : X \in \fD_1\},
\]
and 
\[
\fV_2 = \{(X,X^3) : X \in \fD_1\}.
\]
These varieties are matrix spanning and are biholomorphic via the composition $(X,X^2) \mapsto X \mapsto (X,X^3)$, and it is clear that 
\[
H^\infty(\fV_1) \cong H^\infty(\fV_2) \cong H^\infty(\fD_1) = H^\infty(\bD)
\]
completely isometrically and weak-* continuously. 
However, the isomorphism is not induced by an automorphism of $\fD_2$, because $\Aut(\fD_2) = \Aut(\bD_2)$ and no such automorphism maps one of these varieties onto the other (see \cite[Example 3.4]{MT16}).
\end{example}

\begin{remark}
The above example and the preceding remark display another intriguing difference between the nc ball and the nc polydisk. 
Whereas \cite[Theorem 5.7]{BS+} shows that isomorphisms between the quotients $H^\infty(\fB_d)/I(\fV)$ (for $\fV \subseteq \fB_d$ matrix spanning) can be lifted to an automorphism of $H^\infty(\fB_d)$, this is not so for analogous quotients of $H^\infty(\fD_d)$. 
\end{remark}


\section{Representations fibered over the ball: difficulties and ideas}\label{sec:rep_ball}

Recall that we have a projection $\pi \colon \Rep^{cc}(H^\infty(\bD_Q)) \to \ol{\bD_Q}$ given by
\[
\pi(\Phi) = \Phi(Z) =  (\Phi(Z_1), \ldots, \Phi(Z_d)). 
\]
In the case of the row ball $\bD_Q = \fB_d$, it is known that for every $X \in \bD_Q$ the fiber over $X$ is a singleton, namely
\be\label{eq:pi_inv}
\pi^{-1}(X) = \{\Phi_X\}  \, \textrm{ for all } X \in \fB_d,
\ee
where $\Phi_X$ is the evaluation representation $\Phi_X(f) = f(X)$, and that $\pi^{-1}$ restricts to a homeomorphism from $\bD_Q(k)$ into the space of completely contractive and weak-* continuous $k$-dimensional representations; this is due to Davidson and Pitts who formulated it in the language of the {\em noncommutative Toeplitz algebra} \cite[Theorem 3.2]{DP98b} (it is crucial here that $d<\infty$ \cite{DPErr}; see also \cite[Theorem 6.1]{SSS18} for a formulation in the language of nc functions). 
This readily implies that for a nc subvariety $\fV \subset \fB_d$, a completely contractive finite dimensional representation $\Phi$ of $H^\infty(\fV)$ is weak-* continuous if $\pi(\Phi) = \Phi(Z) = X \in \fV$, in which case $\Phi = \Phi_X$ \cite[Theorem 6.3]{SSS18}. 
These observations lie at the heart of the classification of of the algebras of the form $H^\infty(\fV)$, where $\fV$ is a subvariety of the row ball $\fB_d$, up to completely isometric isomorphism \cite{SSS18} and up to completely bounded isomorphism \cite{SSS20}.

Significantly, in the completely isometric category, as well as in the completely bounded category when restricting to homogeneous subvarieties, one did not require to assume that an isomorphism is weak-* continuous in order to conclude that the subvarieties are nc biholomorphic. 
In fact, it is shown that a completely contractive isomorphism $\alpha \colon H^\infty(\fV) \to H^\infty(\fW)$ is implemented as a composition operator $\alpha(f) = f\circ G$, where $G \colon \fW \to \fV$ is a nc biholomorphism, and it follows that $\alpha$ must therefore be a weak-* homeomorphism --- automatically. 
However, in the setting of the current paper we could not solve the problem whether \eqref{eq:pi_inv} holds in the generality of nc operator balls $\bD_Q$ in place of $\fB_d$.

Our goal in this section is to discuss possible approaches to proving \eqref{eq:pi_inv} in the generality that we are working in this paper, to explain what difficulties arise when working outside of the relatively well-understood and well-behaved row ball $\fB_d$, and to obtain partial positive results when possible.

\subsection{Realizations-based approach}

By a theorem of Agler and McCarthy \cite{AM15a} (which was extended to greater generality by Ball, Marx and Vinnikov \cite{BMV18}), every $f \in H^\infty(\bD_Q)$ of norm $1$ can be expressed concretely via the following {\em realization formula}
\be\label{eq:realizationX}
f(X) = A \otimes {I}_{n} + (B \otimes {I}_{n})(I_M \otimes Q(X))\left[1 - (D \otimes I_n)(I_M \otimes Q(X)) \right]^{-1} (C \otimes I_n)
\ee
for all $X \in \bD_Q(n)$. 
Here, $1$ denotes the identity operator $1 = I_M \otimes I_H \otimes I_n$ and $A \in \bC$, $B \in \cL(M \otimes H, \bC)$, $C \in \cL(\bC, M \otimes H)$ and $D \in \cL(M\otimes H)$ are such that the operator 
\be\label{eq:V}
V = \left[\begin{matrix}A & B \\ C & D  \end{matrix}\right] \colon \bC \oplus (M \otimes H) \to \bC \oplus (M \otimes H)
\ee
is an isometry. 
Let us assume for this discussion that $M$ and $H$ are finite dimensional. 
Thus, as a nc function we can write formally
\be\label{eq:realization}
f(Z) = A + B(I_M \otimes Q(Z))\left[1 - D(I_M \otimes Q(Z)) \right]^{-1} C. 
\ee
Let $X \in \bD_Q$ and let $\Phi$ be a representation such that $\pi(\Phi) = \Phi(Z) = X$. 
Since $Z \mapsto B(I_M \otimes Q(Z))$ is a $\cL(H)$-valued linear polynomial, we have 
\[
\Phi(B(I_M \otimes Q(Z))) = (B \otimes I_n) (I_M \otimes Q(X)). 
\]
Assume for a moment that the $\cL(\bC,M \otimes H)$-valued nc function given by 
\[
Z \mapsto \left[1 - D(I_M \otimes Q(Z)) \right]^{-1} C
\]
is bounded on $\bD_Q$. 
Since $1 - D(I_M \otimes Q(Z))$ is a linear operator-valued polynomial which is invertible on $\bD_Q$, and since representations take inverses to inverses, we have 
\[
\Phi(\left[1 - D(I_M \otimes Q(Z)) \right]^{-1} C) = \left[1 - (D \otimes I_n)(I_M \otimes Q(X)) \right]^{-1} (C \otimes I_n). 
\]
Thus, {\em if} $Z \mapsto \left[1 - D(I_M \otimes Q(Z)) \right]^{-1} C$ is bounded on $\bD_Q$, then we can use the multiplicity of $\Phi$ to obtain
\begin{align*}
\Phi(f) &= \Phi(B(I_M \otimes Q(Z))) \Phi(\left[1 - D(I_M \otimes Q(Z)) \right]^{-1} C) \\
&= (B \otimes I_n) (I_M \otimes Q(X))\left[1 - (D \otimes I_n)(I_M \otimes Q(X)) \right]^{-1} (C \otimes I_n) \\
&= f(X), 
\end{align*}
that is, $\Phi(f) = \Phi_X(f)$. 

To summarize: {\em If $\dim H < \infty$ and if the function $f \in H^\infty(\bD_Q)$ has a realization \eqref{eq:realization} with $\dim M < \infty$, and if, in addition, the nc function $Z \mapsto \left[1 - D(I_M \otimes Q(Z)) \right]^{-1} C$ is bounded on $\bD_Q$, then $\Phi(f) = f(X) = \Phi_X(f)$ for every $X \in \bD_Q$ and every $\Phi \in \pi^{-1}(X)$. }

Unfortunately, even when $M$ and $H$ are finite dimensional it may happen that the nc function $Z \mapsto \left[1 - D(I_M \otimes Q(Z)) \right]^{-1} C$ is unbounded, as the following example shows. 

\begin{example}\label{ex:unbdd}
Consider the case where $d = 2$, $H = \bC^2$, $Q(Z) = \diag(Z_1, Z_2)$, so that $\bD_Q = \fD_2$ is the nc bidisk. 
Let $f$ be the function \eqref{eq:realization} where we take $M = \bC$ and 
\[
A = 0, 
B = \left[ \begin{matrix} 1/\sqrt{2} & 1/\sqrt{2} \end{matrix} \right], 
C = \left[ \begin{matrix} 1/\sqrt{2} \\ 1/\sqrt{2} \end{matrix} \right], 
D = \left[ \begin{matrix} 1/2 & -1/2 \\-1/2 & 1/2 \end{matrix} \right]. 
\]
One readily computes that for a scalar point $x = (x_1, x_2) \in \bD^2 = \fD_2(1)$, 
\[
f(x) = \frac{2x_1 x_2 - x_1 - x_2}{x_1 + x_2 -2}, 
\]
a bounded rational function on $\bD^2$, which can be verified directly and also follows from the fact that the matrix $V$ in \eqref{eq:V} is a unitary. 
However, 
\[
\left[1 - D Q(x) \right]^{-1} C = \sqrt{2} \left[ \begin{matrix} \frac{x_2 - 1}{x_1 + x_2 - 2}  \\ \\ \frac{x_1 - 1}{x_1 + x_2 - 2}  \end{matrix} \right]
\]
which is unbounded on $\bD^2$. 
\end{example}

\begin{remark}
The bad behavior displayed in the previous example cannot occur, at least not for rational functions, in the setting of the row ball $\fB_d$. 
Indeed, Jury, Martin and Shamovich showed that a bounded nc rational function on $\fB_d$ has a realization where the block $D$ has ``spectral radius" strictly less than one, whence the inverse can be computed as a norm convergent Neumann series and, in particular, does not lead to anything unbounded (see \cite[Theorem A]{JMS21}. Strictly speaking, the results in \cite{JMS21} use the so-called {\em descriptor} realization, whereas we here use the so-called {\em Fornasini–Marchesini} realization; however these are equivalent, see e.g. Section 3 in \cite{AMS24} for an explanation). 
\end{remark}

\subsection{Taylor-Taylor expansions-based approach}

For every $f \in H^\infty(\bD_Q)$ and $n \in \bN$ the Taylor-Taylor (TT) expansion of order $N-1$ of $f$ centered at $0$ reads
\[
f(X) = \sum_{k=0}^{N-1} \Delta^k f(0^{(n)}, \ldots, 0^{(n)})[X, \ldots, X] + \Delta^{N}f(0^{(n)}, \ldots, 0^{(n)}, X)[X, \ldots, X], 
\]
for all $X \in \bD_Q(n)$ (see \cite[Theorem 4.1]{KVV14}). 
For every $k = 0, 1, \ldots, N$ the term $\Delta^k f$ is the $k$-th order difference-differential operator applied to $f$, which is a nc function of order $k$ obtained by evaluating $f$ on certain upper triangular matrices. 

Since $0^{(n)}$ is a scalar point, the properties of the higher order nc difference-differential operators imply that the TT expansion can be written in the more convenient and transparent form 
\be\label{eq:TT}
f(Z) = \sum_{|w|<N} Z^w \Delta^{w} f(0,\ldots, 0) + \sum_{|w|=N} Z^w \Delta^{w} f(0,\ldots, 0,Z), 
\ee
where $Z^w$ are just the monomials $Z_{w_1} Z_{w_2} \cdots Z_{w_k}$ (with $k = |w|$), $\Delta^{w} f(0,\ldots, 0)$ are scalar coefficients and $\Delta^{w} f(0,\ldots, 0,Z)$ is a nc function in $Z$ (see \cite[Corollary 4.4]{KVV14} for details). 
We note that the terms $Z^w \Delta^{w} f(0,\ldots, 0,Z)$ are simply the monomial $Z^w$ multiplied from the right by a certain nc function which just happens to arise as a certain difference-differential operator $\Delta^w$ of order $N$ applied to $f$ and evaluated at the points $0, \ldots, 0, Z$. 

Now, let $X \in \bD_Q$ and let $\Phi \in \pi^{-1}(X)$. 
If, for all words $w$ of length $N$, the function $\Delta^wf(0,\ldots,0,Z)$ happens to be in $H^\infty(\bD_Q)$ then we can apply $\Phi$ to \eqref{eq:TT} to obtain
\[
\Phi(f) = \sum_{|w|<N} X^w \Delta^{w} f(0,\ldots, 0) + \sum_{|w|=N} X^w \Phi(\Delta^{w} f(0,\ldots, 0,Z)). 
\]
One can then try to control the remainder term, using the fact that $X$ is an interior point so that $\|X^w\|$ should be small for large $|w|$ and that $\Phi$ is completely contractive. 
The question then arises, whether the terms $\Delta^w f (0,\ldots, 0,Z)$ can be bounded, say by some constant times $\|f\|$. 
Somewhat surprisingly, it turns out that this is not the case. 

\begin{example}\label{ex:polydisk_blowup}
Let $f$ be the bounded nc function on $\fD_2$ as in Example \ref{ex:unbdd}. 
We shall show that already its first difference-differential $\Delta f(0,X)$ is unbounded on $\fD_2$. 
In fact, $\Delta f(0,X)$ is unbounded on $\bD^2$. 
To this end, we apply $f$ to the $2$-tuple $X = (X_1, X_2) \in M_2^2$ given by 
\[
X_1 = \left[\begin{matrix} 0 & h_1 \\ 0 & x_1 \end{matrix}\right] \quad, \quad X_2 = \left[\begin{matrix} 0 & h_2 \\ 0 & x_2 \end{matrix}\right] , 
\]
or in other words 
\[
X = \left[\begin{matrix} 0 & h \\ 0 & x \end{matrix}\right]
\]
where $x = (x_1, x_2) \in \bD^2$ and $h = (h_1, h_2) \in \bC^2$ is small enough so that $X \in \fD_2$. 
The nc function $f$ applied to $X$ gives by nc difference-differential calculus (see \cite[Section 2.2]{KVV14}): 
\[
f(X) = \left[\begin{matrix} f(0) & \Delta f(0,x) [h] \\ 0 & f(x) \end{matrix}\right] = \left[\begin{matrix} f(0,0) & \Delta_1 f(0,x) h_1 + \Delta_2 f(0,x) h_2 \\ 0 & f(x_1,x_2) \end{matrix}\right] .
\]
On the other hand, plugging $X$ into the realization formula \eqref{eq:realizationX} we find that 
\[
f(X) = \left[\begin{matrix} 0 & & \frac{x_2 - 1}{x_1 + x_2 - 2}h_1 +  \frac{x_1 - 1}{x_1 + x_2 - 2}h_2  \\ & & \\ 0 & & \frac{x_1 + x_2  - 2x_1 x_2}{x_1 + x_2 - 2} \end{matrix}\right].
\]
Comparing the above expressions, one can read off that
\[
\Delta_1 f(0,x) = \frac{x_2 - 1}{x_1 + x_2 - 2},
\]
which is clearly unbounded on $\bD^2$. 
\end{example}

\begin{remark}
This example goes counter to the intuition acquired from the case $d=1$, where the difference-differential operator evaluated at scalar points is just 
\[
\Delta f(x,y) = \frac{f(x) - f(y)}{x-y},
\]
which is bounded for $x \in \bD$ when $y \in \bD$ is held fixed. 
It is worth noting that this behavior is related to the nc TT expansion, even though unboundedness occurs already at the scalar level. 
For, what we have seen is that in the nc TT expansion around the origin of order zero 
\[
f(x) = f(0) + \Delta_1 f(0,x) x_1 + \Delta_2 f(0,x) x_2, 
\]
the factors $\Delta_i f(0,x)$ in the remainder term may be unbounded. 
On the other hand, it is well known that if $f \in H^\infty(\bD^2)$, then there exist $g_1, g_2 \in H^\infty(\bD^2)$ so that
\[
f(x) = f(0) + g_1(x) x_1 + g_2(x) x_2.
\]
Indeed, as explained in \cite[Section 6.6.1]{Rud08}, this is an easy version of Gleason's problem; take 
\[
g_1(x) = \frac{f(x_1, 0) - f(0,0)}{x_1} \quad, \quad g_1(x) = \frac{f(x_1, x_2) - f(x_1,0)}{x_2} .
\]
\end{remark}

\subsection{Some positive results}

The idea to use \eqref{eq:TT} in order to show that representations fibered over points $X \in \bD_Q$ are necessarily point evaluations comes directly from Davidson and Pitt's proof of \cite[Theorem 3.2]{DP98b}. 
Proposition 2.6 in \cite{DP98b} provides a Taylor expansion in operator-theoretic language suited for the noncommutative analytic Toeplitz algebra, which is equivalent\footnote{The identification of $H^\infty(\fB_d)$ with the nc analytic Toeplitz algebra is explained in \cite[Section 4]{SSS18}.} to \eqref{eq:TT} for functions in $H^\infty(\fB_d)$. 
In this subsection we draw inspiration from their proof to obtain positive results for a class of nc operator balls $\bD_Q$. 

Let $\cJ_N$ denote the weak-* closed ideal in $H^\infty(\fB_d)$ generated by the monomials of length $N$, i.e., $\{Z^w : |w|=N\}$. 

\begin{definition}\label{def:regular}
A nc operator ball $\bD_Q$ is said to be {\em right regular} if for every $f \in H^\infty(\bD_Q)$ and for every word $w$ the nc function $\Delta^w f (0,\ldots, 0,X)$ is in $H^\infty(\bD_Q)$ and if there exists a constant $M$ such that for every $N$ and for every $f \in \cJ_N$
\[
\left( \sup_{X \in \bD_Q}  \left\|\operatorname{row}(X^w)_{|w|=N} \right\| \right)  \cdot \left( \sup_{X \in \bD_Q} \left \| \operatorname{col}(\Delta^w f(0,\ldots, 0,X))_{|w|=N} \right \| \right) \leq M \|f\|. 
\]
\end{definition}

\begin{example}
The nc row ball $\fB_d$ is right regular. 
Indeed, if $f \in \cJ_N$ the TT expansion of order $N$ consists of the remainder term alone which we write as 
\[
f(Z) = \sum_{|w|=N} Z^w f_w(Z) = \operatorname{row}(Z^w)_{|w|=N} \cdot \operatorname{col}(f_w(Z))_{|w|=N}. 
\]
But $Z$ is a row isometry, and therefore 
\[
\|f\| = \|\operatorname{row}(Z^w)_{|w|=N} \cdot \operatorname{col}(f_w)_{|w|=N}\| = \|\operatorname{row}(Z^w)_{|w|=N} \| \| \operatorname{col}(f_w)_{|w|=N}\| = \|\operatorname{col}(f_w)_{|w|=N}\|. 
\]
In particular $\Delta^w f (0,\ldots, 0,\cdot) = f_w \in H^\infty(\fB_d)$ for all $w$ and $\fB_d$ is right regular. 
\end{example}

The above argument shows that $\bD_Q$ is regular whenever the nc operator ball is such that $Z$ is a row isometry (meaning that in any completely isometric embedding $A(\bD_Q)$ into a C*-algebra, the images of $Z_1, \ldots, Z_d$ are isometries with pairwise orthogonal ranges). 
However, this observation does not really lead to new examples. 

\begin{proposition}
If the coordinate functions $Z_1,\ldots, Z_d \in A(\bD_Q)$ form a row isometry, then $\bD_Q = \fB_d$. 
\end{proposition}
\begin{proof}
Popescu showed that all unital operator algebras generated by a row isometry of the same length are completely isometrically isomorphic \cite{Pop96}; 
thus, $A(\bD_Q)  \cong A(\fB_d)$. 
By \cite[Corollary 6.18]{SampatShalit25}, there is a bijective linear transformation $T \colon \bC^d \to \bC^d$ such that $T(\fB_d) = \bD_Q$ and such that the completely isometric isomorphism is given by
\[
A(\bD_Q) \ni f \mapsto f \circ T \in A(\fB_d).
\]
Now if we equip both copies of $\bC^d$ with the Euclidean norm, we see that $T$ must be a contraction, for otherwise
\[
\|Z\|_{A(\bD_Q)} = \sup_{X \in \bD_Q}\|X\| \geq \sup_{x \in T(\bB_d)} \|x\| > 1,
\]
and then $Z$ is not even a row contraction, not to mention a row isometry. 
It follows that $\bD_Q \subseteq \fB_d$, because every $Y \in \bD_Q$ has the form $Y = X (I_n \otimes T^t)$ for $X \in \fB_d$, which is a product of contractions, so $Y$ is a row contraction. 

Now, $T$ breaks up as the direct sum $T = A \oplus U$ where $U$ is a unitary and $A$ is a strict contraction, where either one of them might be nil (i.e., not present). 
If $A$ is nil, then we are done, for then $\bD_Q$ is a unitary image of $\fB_d$, whence $\bD_Q = \fB_d$ as required. 
If $A$ is not nil then, without loss of generality, we may assume there is an orthonormal basis $\{e_1, \ldots, e_d\}$ in which $T$ is represented by an upper triangular matrix, so that the $d$-th entry in $T(\sum \alpha_j e_j)$ is equal to $a\alpha_d e_d$ for $a \in \bD$. 
But then the coordinate function $Z_d$ on $\bD_Q$ satisfies $\|Z_d\| = 1$ (because it is an isometry), while
\[
\|Z_d \circ T \| = \sup_{X \in \fB_d} \|Z_d(X (I_n \otimes T^t))\| = \sup_{X \in \fB_d} |a| \|X_d\| = |a| < 1, 
\]
so $f \mapsto f \circ T$ is not an isometry, in contradiction to the first observation in the proof. 
\end{proof}

\begin{example}
Example \ref{ex:polydisk_blowup} shows that $\fD_2$ is not right regular, and in fact does not satisfy any imaginable weaker regularity property since $\Delta f(0,X)$ need not be levelwise bounded. 
In fact, $\fD_2$ is the maximal nc operator ball with first level equal to $\bD^2$, but since $\Delta f(0,X)$ blows up already on the first level, it follows that {\em no nc operator ball over $\bD^2$ is right regular}.
\end{example}

\begin{example}
We do not know whether the column unit ball 
\[
\fC_d = \left\{X \in \bM^d : \left\|\sum_{j=1}^d X_j^* X_j\right\| < 1\right\}, 
\]
is right regular, but it is easy to see that it is {\em left} regular (in an obvious sense). 
This can be seen by noting that by applying left difference-differential calculus instead of right difference-differential calculus, we can write the remainder term for $f \in \cJ_N$ as 
\[
f(Z) = \sum_{|w|=N} g_w(Z) Z^w = \operatorname{row}(g_w(Z))_{|w|=N} \cdot \operatorname{col}(Z^w)_{|w|=N}. 
\]
Since the tuple $Z = (Z_1, \ldots, Z_d)$ on $\fC_d$ is a column coisometry, we can argue as above and find $\|\operatorname{row}(g_w(Z))_{|w|=N}\| = \|f\|$, etc. 
\end{example}

Finally, we present a result in the positive direction for right regular nc operator balls, following the proof of \cite[Theorem 3.2]{DP98b}. 
We leave open for future research the question of which balls are regular. 

\begin{proposition}
If $\bD_Q$ is right regular, then, for every $X \in \bD_Q$, there is a unique completely contractive representation $\Phi$ of $H^\infty(\bD_Q)$ such that $\pi(\Phi) = X$, namely, $\Phi = \Phi_X$. 
\end{proposition}
\begin{proof}
Let $X \in \bD_Q$, so that $\|Q(X)\| = r < 1$, and suppose that $\Phi$ is completely contractive and that $\pi(\Phi) = X$. 
For $f \in H^\infty(\bD_Q)$ we let
\[
\Sigma_N(f) = \sum_{0\leq k<N} \left(1 - \frac{k}{N}\right) f_k
\]
be the Ces\`{a}ro sums which converge bounded pointwise to $f$ as $N \to \infty$. 
We shall show that
\[
\lim_{N \to \infty}\Phi(\Sigma_N(f)) = \Phi(f)
\]
Once we show this, we shall be done, because $\Phi$ and $\Phi_X$ agree on polynomials, so using the weak-* continuity of $\Phi_X$, we will obtain 
\[
\Phi(f) = \lim_{N \to \infty} \Phi(\Sigma_N(f)) = \lim_{N \to \infty} \Phi_X(\Sigma_N(f)) = \Phi_X(f). 
\]
To show that
\[
\lim_{N \to \infty} \Phi(\Sigma_N(f)) = \Phi(f),
\]
fix $\epsilon > 0$ and let $m$ be so that $r^m < \epsilon$. 
Now, for all $N$ sufficiently large, we have
\[
\left \|\sum_{k<m}\frac{k}{N} f_k \right\| < \epsilon. 
\]
We can therefore write 
\[
f - \Sigma_N(f) = g + \sum_{k<m}\frac{k}{N} f_k , 
\]
where $g \in \cJ_m$ satisfies 
\[
\|g \| \leq  \|f\| + \|\Sigma_N(f)\| + \left\|\sum_{k<m}\frac{k}{N} f_k\right\| < 2\|f\| + \epsilon 
\]
for all sufficiently large $N$. 
Therefore 
\[
\|\Phi(f) - \Phi(\Sigma_N(f))\| \leq \|\Phi(g)\| + \left\|\Phi\left( \sum_{k<m}\frac{k}{N} f_k \right) \right\| < \|\Phi(g)\| + \epsilon. 
\]
It remains to estimate $\|\Phi(g)\|$. 
Since $\|Q(X)\|=r$, we have that $X = rY$ for $Y \in \partial \bD_Q$. 
Now 
\[
g(Z) = \sum_{|w|=m} Z^w f_w(Z), 
\]
where we denote $f_w(Z) = \Delta^w f(0,\ldots, 0, Z)$, and by assumption of right regularity, all the $f_w$ are in $H^\infty(\bD_Q)$. 
Thus,
\[
\Phi(g) = \sum_{|w|=m} X^w \Phi(f_w) = r^m \sum_{|w|=m} Y^w \Phi(f_w).
\]
Since $\Phi$ is completely contractive and $\bD_Q$ is right regular, 
\begin{align*}
\|\Phi(g)\| 
&\leq r^m \left \| \operatorname{row}(Y^w)_{|w|=N} \right \| \left\| \operatorname{col}(\Phi(f_w))_{|w|=N} \right\| \\
& \leq r^m \left \| \operatorname{row}(Z^w)_{|w|=N} \right \| \left\| \operatorname{col}(f_w)_{|w|=N} \right\| \\
& \leq r^m M \|g\| .
\end{align*}
We conclude that 
\[
\|\Phi(g)\| \leq r^m M \|g\| \leq r^m M(2\|f\| + \epsilon) < \epsilon M(2\|f\| + \epsilon), 
\]
showing that
\[
\lim_{N \to \infty} \|\Phi(f) - \Phi(\Sigma_N(f))\| = 0,
\]
as required. 
\end{proof}

Finally, using \eqref{eqn:rep.on.V.in.terms.of.rep.on.D_Q}, we conclude the following. 
\begin{corollary}
Let $\bD_Q$ be a right regular nc operator ball and $\fV \subset \bD_Q$ a nc variety. 
Then for every $X \in \fV$, there is a unique completely contractive representation $\Phi$ of $H^\infty(\fV)$ such that $\pi(\Phi) = X$, namely, $\Phi = \Phi_X$. 
\end{corollary}

\subsection*{Acknowledgements} The authors would like to thank M. Hartz for suggesting the idea for Theorem \ref{theorem:predual.H^infty(D_Q)}. Thanks also to G. Knese for his help on obtaining Example \ref{ex:unbdd}. 
Finally, the authors would like to thank the anonymous referees for their thoughtful and helpful reports. 


\bibliographystyle{abbrv}
\bibliography{bibliography}

\end{document}